\documentclass[10pt]{amsart}
\input epsf

\usepackage{amsfonts,amsthm,amsmath,amssymb,latexsym}
\usepackage{graphicx,color}
\usepackage[all]{xy}
\usepackage{wrapfig}
\usepackage{fullpage}

\address{HH:Department of Mathematics, Department of Mathematics, Kansas State University. Manhattan, KS 66502}
\email{hakobyan@math.ksu.edu}

\address{DS:Department of Mathematics, Queens College of CUNY,
65-30 Kissena Blvd., Flushing, NY 11367}
\email{Dragomir.Saric@qc.cuny.edu}

\address{DS:Mathematics PhD. Program, The CUNY Graduate Center, 365 Fifth Avenue, New York, NY 10016-4309}

\theoremstyle{definition}

 \newtheorem{definition}{Definition}[section]
 \newtheorem{remark}[definition]{Remark}

\theoremstyle{plain}

 \newtheorem{proposition}[definition]{Proposition}
 \newtheorem{theorem}[definition]{Theorem}
 \newtheorem{corollary}[definition]{Corollary}
 \newtheorem{lemma}[definition]{Lemma}

\newcommand{\eps}{\varepsilon}
\newcommand{\G}{\Gamma}
\newcommand{\g}{\gamma}
\newcommand{\D}{\Delta}

\renewcommand{\d}{\delta}
\newcommand{\dist}{\mathrm{dist}}
\newcommand{\m}{\mathrm{mod}}
\renewcommand{\O}{\Omega}
\newcommand{\diam}{\mathrm{diam}}

\title[Limits of geodesics in the Universal Teichm\"uller space]{Limits of Teichm\"uller geodesics in the Universal Teichm\"uller space}
\author{Hrant Hakobyan and Dragomir \v Sari\' c}
\thanks{The second author was partially supported by National Science Foundation grant DMS 1102440 and by the Simons Foundation grant.}

\begin{document}


\date{\today}

\begin{abstract}
Thurston's boundary to the universal Teichm\"uller space $T(\mathbb{H})$ is
the set of asymptotic rays to the embedding of $T(\mathbb{H})$ in the space
of geodesic currents; the boundary is identified with the projective
bounded measured laminations $PML_{bdd}(\mathbb{H})$ of $\mathbb{H}$. We
prove that each Teichm\"uller geodesic ray in $T(\mathbb{H})$ has a unique
limit point in Thurston's boundary to $T(\mathbb{H})$ unlike in the
case of closed surfaces.
\end{abstract}

\maketitle

\section{Introduction}

The Teichm\"uller space $T(\mathbb{D})$ of the unit disk $\mathbb{D}$, called
the {\it universal} Teichm\" uller space, consists of all quasisymmetric maps
$h:\mathbb{S}^1\to \mathbb{S}^1$ which fix $1$, $i$ and $-1$ on the unit
circle $\mathbb{S}^1$ (cf. \cite{GL}). The Teichm\" uller space of an
arbitrary hyperbolic surface embeds in $T(\mathbb{D})$ as a complex Banach
submanifold. Thurston's boundary to the universal Teichm\"uller space
$T(\mathbb{D})$ is the space of projective bounded measured laminations
$PML_{bdd}(\mathbb{D})$ of $\mathbb{D}$ (cf. \cite{Sa2}, \cite{Sa3}). We
study the limits of Teichm\"uller geodesic rays on Thurston's boundary to
$T(\mathbb{D})$.

Bonahon \cite{Bon1} defined an embedding of the Teichm\"uller space $T(S)$ of
a closed surface $S$ of genus at least two into the space of geodesic
currents (equipped with the weak* topology). The space of asymptotic rays to
the image of the embedding of $T(S)$ is identified with the space of
projective measured lamination of $S$-Thurston's boundary to $T(S)$. The
universal Teichm\" uller space $T(\mathbb{D})$ embeds into the space of
geodesic currents of $\mathbb{D}$ when geodesic currents are equipped with
the {\it uniform} weak* topology (cf. \cite{Sa2}, \cite{MiSar}, \cite{Sa1})
and this embedding is real analytic (cf. Otal \cite{Ot}). The image of
$T(\mathbb{D})$ in the space of geodesic currents is closed and unbounded,
and the space of asymptotic rays to the image of $T(\mathbb{D})$-Thurston's
boundary to $T(\mathbb{D})$- is identified with the projective bounded
measured laminations $PML_{bdd}(\mathbb{D})$ (cf. \cite{Sa2}, \cite{Sa1}). In
particular, the earthquake paths $t\mapsto E^{t\mu}|_{\mathbb{S}^1}$ as
$t\to\infty$ accumulate to their corresponding projective earthquake measures
$[\mu ]\in PML_{bdd}(\mathbb{D})$ in the uniform weak* topology (cf.
\cite{Sa2}, \cite{Sa1}). The construction of Thurston's boundary works for
all hyperbolic surfaces simultaneously since any invariance under a Fuchsian
group is preserved under the construction.

In the case of closed surfaces, Masur \cite{Mas} proved that the
Teichm\"uller geodesic rays obtained by shrinking the vertical trajectories of
holomorphic quadratic differentials with uniquely ergodic vertical foliations
converge to the projective classes of their vertical foliations in
Thurston's boundary. On the other hand, if the vertical foliation of a
holomorphic quadratic differential consists of finitely many cylinders then
the limit of the Teichm\"uller geodesic on Thurston's boundary is the projective class of the measured
lamination supported on finitely many simple closed geodesics homotopic  to
the cylinders with equal weights (cf. \cite{Mas}).
 In both cases the Teichm\"uller geodesic rays have unique endpoints on Thurston's boundary and the endpoints depend only on the vertical foliations. However, when the vertical foliations of holomorphic quadratic differentials on closed surfaces are not uniquely ergodic then the limit sets of the corresponding Teichm\"uller rays consist of more than one point (cf. Lenzhen \cite{Len}, Leininger-Lenzhen-Rafi \cite{LLR}, and Chaika-Masur-Wolf \cite{CMW}).

 We consider the limits of Teichm\"uller geodesic rays in the universal Teichm\"uller space $T(\mathbb{D})$ corresponding to integrable holomorphic quadratic differentials. In our previous work we showed that when a holomorphic quadratic differential has no zeroes in $\mathbb{D}$, and the natural parameter maps the unit disk onto a domain in $\mathbb{C}$ between the graphs of two functions, then the Teichm\"uller geodesic ray has a unique endpoint on Thurston's boundary of $T(\mathbb{D})$, but the endpoint depends on both vertical and horizontal foliations of $\varphi$ (cf. \cite{HaSar}). We extend this result to all integrable holomorphic quadratic differentials on the unit disk $\mathbb{D}$.

The hyperbolic plane is identified with the unit disk $\mathbb{D}$ and the
visual boundary of the hyperbolic plane is identified with the unit circle
$\mathbb{S}^1$. A (hyperbolic) geodesic in $\mathbb{D}$ is uniquely
determined by it endpoints; the space of geodesics of $\mathbb{D}$ is
identified with $\mathbb{S}^1\times \mathbb{S}^1-diag$.
 Let $\varphi$ be an arbitrary integrable holomorphic quadratic differential on the unit disk $\mathbb{D}$.
 Each vertical trajectory
 of $\varphi$ has two distinct endpoints on the boundary circle $\mathbb{S}^1$ of the unit disk $\mathbb{D}$ (cf. \cite{Str}).
 Thus each vertical trajectory of $\varphi$ is homotopic to a unique geodesic of $\mathbb{D}$ relative ideal endpoints on $\mathbb{S}^1$. Let $v_{\varphi}$ be the set of the geodesics in $\mathbb{D}$ homotopic to the vertical trajectories of $\varphi$ (cf. Figure 1).
Given a box of geodesics $[a,b]\times [c,d]\subset \mathbb{S}^1\times
\mathbb{S}^1-diag$, denote by $I_{[a,b]\times [c,d]}$ any (at most countable)
union of sub-arcs of horizontal trajectories that intersects exactly once
each vertical trajectory of $\varphi$ with one endpoint in $[a,b]$ and the
other endpoint in $[c,d]$, and that does not intersect any other vertical
trajectories of $\varphi$.

\begin{figure}
\noindent\makebox[\textwidth]{
\includegraphics[width=12 cm]{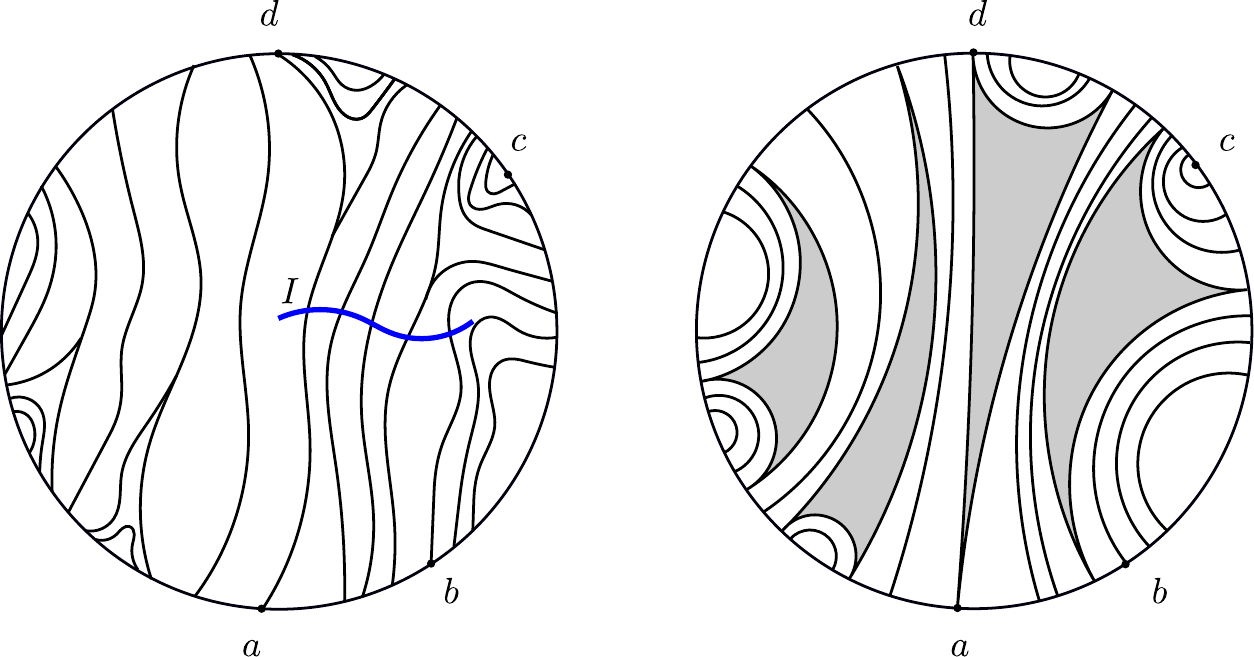}}
\caption{The vertical foliation of $\varphi$ and the correspondingg geodesic lamination $v_{\varphi}$ The measure $\mu_{\varphi}([a,b]\times [c,d])$ is obtained by integration $\int_I\frac{1}{l^{\varphi}(z)}|\sqrt{\varphi (z)}dz|$.}
\end{figure}

 We define a measured lamination $\mu_{\varphi}$ of $\mathbb{D}$ supported on $v_{\varphi}$ by
\begin{align}\label{def:lamination}
 \mu_{\varphi}([a,b]\times [c,d])=\int_{I_{[a,b]\times [c,d]}}\frac{1}{l(x)}dx,
\end{align}
where $l(x)$ is the $\varphi$-length (i.e. the length induced by $\int |\sqrt{\varphi (z)}dz|$) of a vertical trajectory through $x\in
I_{[a,b]\times [c,d]}$ and the integration is in the natural parameter of
$\varphi$.
 We obtain (cf. Proposition \ref{prop:common_endpoints} and proof of Theorem \ref{thm:weak*convergence})

 \vskip .2 cm

\noindent {\bf Proposition 1.} {\it Let $\mu_{\varphi}$ be the measured
lamination homotopic to the vertical foliation of an integrable holomorphic
quadratic differential $\varphi$ on $\mathbb{D}$ defined by the above integration. Then
$$
\|\mu_{\varphi}\|_{Th}=\sup_{[a,b]\times [c,d]}\mu_{\varphi}([a,b]\times [c,d])<\infty
$$
where the supremum is over all boxes of geodesics $[a,b]\times [c,d]\subset
\mathbb{S}^1\times \mathbb{S}^1-diag$ with $cr(a,b,c,d)=2$.

The measured lamination $\mu_{\varphi}$ satisfies
$$
\mu_{\varphi}(\{ a\}\times [c,d])=0
$$
for all $a\in \mathbb{S}^1$ and $[c,d]\subset \mathbb{S}^1$, and in
particular $\mu_{\varphi}$ does not have atoms. }

\vskip .2 cm

For $\epsilon >0$, let $T_{\epsilon}$ be the Teichm\"uller mapping that
shrinks the vertical trajectories of $\varphi$ by a multiplicative constant $\epsilon$.
The {\it Teichm\"uller geodesic ray} $\epsilon \mapsto T_{\epsilon}$ as $\epsilon \to 0^{+}$ leaves every
bounded subset of the universal Teichm\"uller space $T(\mathbb{D})$. We
obtain (cf. Theorem \ref{thm:weak*convergence})

\vskip .2 cm

\noindent {\bf Theorem 1.} {\it Let
$$
\epsilon\mapsto T_{\epsilon}
$$
be the Teichm\"uller geodesic ray in $T(\mathbb{D})$ that shrinks the vertical
trajectories  of an integrable holomorphic quadratic differential $\varphi$
by a multiplicative constant $\epsilon >0$. Then
$$
T_{\epsilon}\to [\mu_{\varphi}]\in PML_{bdd}(\mathbb{D})
$$ as $\epsilon\to 0^{+}$ in Thurston's closure $T(\mathbb{D})\cup PML_{bdd}(\mathbb{D})$ of $T(\mathbb{D})$, where $\mu_{\varphi}$ is the measured lamination defined by equation (\ref{def:lamination}) and the convergence is in the weak* topology on geodesic currents.

In particular, the limit set of any Teichm\"uller ray in $T(\mathbb{D})$
consists of a unique point. }

\vskip .2 cm

\noindent {\bf Remark 1.} The limit point $\mu_{\varphi}$ depends on the
vertical foliation and on the lengths of the vertical trajectories unlike for
closed surfaces. The lengths of vertical trajectories are given by the
transverse measure to the horizontal foliation. Therefore the limit point
depends on both vertical and horizontal foliations of $\varphi$ which is a new phenomenon that does not appear for closed surfaces.

\vskip .2 cm

\noindent {\bf Remark 2.} The measure $\mu_{\varphi}([a,b]\times [c,d])$ is in fact
the conformal modulus of the family of vertical trajectories of $\varphi$ with one endpoint in
$[a,b]$ and another endpoint in $[c,d]$ (cf. Proposition
\ref{prop:vertical_mod_measure}).

\vskip .2 cm

\noindent {\bf Remark 3.} The above theorem is motivated by the results of Masur \cite{Mas} in the case of a closed surface. The major difference in this work is that the hyperbolic plane has no closed geodesics and that the universal Teichm\"uller space is infinite dimensional non-separable Banach manifold. The convergence questions that arise in this setup and the methods applied are of a more analytic nature than for closed surfaces. Moreover, we emphasise the existence of a unique limit point for any Teichm\"uller ray in $T(\mathbb{D})$ which is not true for the Teichm\"uller spaces of closed surfaces.

\vskip .2 cm

The convergence of Teichm\"uller geodesic rays is in the weak* topology while the convergence of earthquake paths is in the {\it uniform} weak* topology. We prove in \S \ref{sec:ce}

\vskip .2 cm

\noindent {\bf Proposition 2.}
{\it There exists an integrable holomorphic quadratic differential $\varphi$ on the unit disk $\mathbb{D}$
such that the corresponding Teichm\"uller geodesic ray does not converge in the {\it uniform} weak* topology.}

\vskip .2 cm

\noindent {\bf Remark 4.} Thus while any Teichm\"uller geodesic ray in $T(\mathbb{D})$ converges in the weak* topology, there exist Teichm\"uller rays that do not converge in the strongest possible sense (the uniform weak* topology) in $T(\mathbb{D})\cup PML_{bdd}(\mathbb{D})$ unlike earthquake paths. Also note that the uniform weak* topology and the weak* topology agree on the space of geodesic currents of a closed surface.

\vskip .2 cm

Denote by $A(\mathbb{D})$ the space of all integrable holomorpic quadratic differentials on the unit disk $\mathbb{D}$. Let $PA(\mathbb{D})=(A(\mathbb{D})-\{ 0\})/\sim$, where $\varphi\sim\varphi_1$ if there exists $c>0$ with $\varphi =c\varphi_1$. By definition, we have $\mu_{c\varphi}=\mu_{\varphi}$ for any $c>0$. Therefore we obtained a map
$$
\mathcal{M} :PA(\mathbb{D})\to PML_{bdd}(\mathbb{D})
$$
given by
$$
\mathcal{M}([\varphi ])=[\mu_{\varphi}],
$$
where $[\varphi ]$ and $[\mu_{\varphi}]$ are the projective classes of $\varphi$ and $\mu_{\varphi}$, respectively.

\vskip .2 cm

\noindent {\bf Theorem 2.} {\it The map
$$
\mathcal{M} :PA(\mathbb{D})\to PML_{bdd}(\mathbb{D});\ \mathcal{M}:[\varphi ]\mapsto [\mu_{\varphi}]
$$
is injective.}

\vskip .2 cm

\noindent {\bf Remark 5.} By Theorem 1 and Theorem 2, two different  Teichm\"uller geodesic rays in $T(\mathbb{D})$ starting at the basepoint of $T(\mathbb{D})$ converge to different points in Thurston's boundary. On the other hand, Masur \cite{Mas} proved that two Teichm\"uller geodesic rays corresponding to two holomorphic quadratic differentials whose vertical foliations decompose a compact surface into finitely many cylinders of the same topological type but different relative heights converge to the same point in Thurston's boundary.

\vskip .2 cm

For $\varphi\in A(\mathbb{D})$, denote by $\nu_{\varphi}$ the measured lamination  whose support is the geodesic lamination $v_{\varphi}$ homotopic to the vertical foliation of $\varphi$ and whose transverse measure is induced by the transverse measure to the vertical foliation induced by $\varphi$. We recover the integral of $|\varphi |$ using $\mu_{\varphi}$ and $\nu_{\varphi}$, namely
$$
\|\varphi\|_{L^1}=\int_{\mathbb{S}^1\times \mathbb{S}^1-diag}\frac{d\nu_{\varphi}}{d\mu_{\varphi}}d\nu_{\varphi}.
$$

\vskip .2 cm

\noindent {\bf Theorem 3.} {\it The map
$$
A(\mathbb{D})\to ML_{bdd}(\mathbb{D})\times ML_{bdd}(\mathbb{D})
$$
defined by
$$
\varphi\mapsto (\nu_{\varphi},\mu_{\varphi})
$$
is injective.}

\vskip .2 cm

The paper is organized as follows. In \S 2 we define the universal Teichm\"uller space $T(\mathbb{D})$, the space of geodesic currents, the Liouville current and Thurston's boundary to $T(\mathbb{D})$.
In \S 3 we define modulus of a family of curves and find a relationship between the modulus and relative distance. Finally we give asymptotic relationship between the modulus and the Liouville current which is fundamental to our work. In \S 4 we study the limits of Teichm\"uller geodesic rays and prove Theorem 1. In \S 5 we give a counter-example to uniform weak* convergence of Teichm\"uller geodesic rays. In \S 6 we study the relationship between the integrable holomorphic quadratic differentials and two measured laminations homotopic to the vertical foliation.

\vskip .2 cm

{\it Acknowledgements.} We would like to thank Mr. Huiping Pan for his useful comments.

\section{Thurston's boundary via geodesic currents}

We identify the unit disk $\mathbb{D}$ with the hyperbolic plane; the visual
boundary to $\mathbb{D}$ is the unit circle $\mathbb{S}^1$. A homeomorphism
$h:\mathbb{S}^1\to \mathbb{S}^1$ is said to be {\it quasisymmetric} if there
exists $M\geq 1$ such that
$$
\frac{1}{M}\leq\frac{|h(I)|}{|h(J)|}\leq M
$$
for all circular arcs $I,J$ with a common boundary point and disjoint
interiors such that $|I|=|J|$, where $|I|$ is the length of $I$. A
homeomorphism is quasisymmetric if and only if it extends to a quasiconformal
map of the unit disk, see e.g. \cite{LV}.

\begin{definition}
The universal Teichm\"uller space $T(\mathbb{D})$ consists of all
quasisymmetric maps $h:\mathbb{S}^1\to \mathbb{S}^1$ that fix $-i,1,i\in
\mathbb{S}^1$.
\end{definition}

If $g:\mathbb{D}\to\mathbb{D}$ is a quasiconformal map, denote by $K(g)$ its
quasiconformal constant. The Teichm\"uller metric on $T(\mathbb{D})$ is given
by $d(h_1,h_2)=\inf_g \log K(g)$, where $g$ runs over all quasiconformal
extensions of the quasisymmetric map $h_1\circ h_2^{-1}$. The Teichm\"uller
topology is induced by the Teichm\"uller metric.

Bonahon's approach \cite{Bon1} to Thurston's boundary of the Teichm\"uller
space $T(S)$ of a closed surface $S$ is to embed $T(S)$ into the space of
geodesic currents on $S$. A {\it geodesic current} on $S$ is a positive Borel
measure on the space of geodesics $\mathbb{S}^1\times \mathbb{S}^1- diag$ of
the universal covering $\mathbb{D}$ of $S$ that is invariant under the action
of the covering group $\pi_1(S)$. Each point in the Teichm\"uller space
$T(S)$ is a quasisymmetric map
$$
h:\mathbb{S}^1\to \mathbb{S}^1
$$
that conjugates the covering Fuchsian group $\pi_1(S)$ onto another Fuchsian
group.

 The {\it Liouville measure} $\mathcal{L}$ on the space of geodesic of $\mathbb{D}$ is given by
 $$
 \mathcal{L}(A)=\int_A \frac{d\alpha d\beta}{|e^{i\alpha}-e^{i\beta}|^2}
 $$
 for any Borel set $A\subset \mathbb{S}^1\times \mathbb{S}^1-diag$. If $A=[a,b]\times [c,d]$ is a {\it box of geodesics} then
 $$
 \mathcal{L}([a,b]\times [c,d])=\log\frac{(a-c)(b-d)}{(a-d)(b-c)}.
 $$

To each quasisymmetric map $h:\mathbb{S}^1\to \mathbb{S}^1$ that conjugates
$\pi_1(S)$ onto another Fuchsian group, we assign the pull back
$h^{*}(\mathcal{L})$ of the Liouville measure. This assignment is a
homeomorphism of $T(S)$ onto its image in the space of geodesic currents for
$S$ when equipped with the weak* topology (cf. \cite{Bon1}). The set of the
asymptotic rays to the image of $T(S)$ is identified with the projective
measured laminations on $S$-the Thurston's boundary of $T(S)$ (cf.
\cite{Bon1}).

The universal Teichm\"uller space $T(\mathbb{D})$ maps into the space of geodesic currents
by taking the pull backs by quasisymmetric maps of the Liouville measure.
There is no invariance condition on the quasisymmetric maps or on the pull
backs of the Liouville measure. A geodesic current $\alpha$ is {\it bounded}
if
$$
\sup_{[a,b]\times [c,d]}\alpha ([a,b]\times [c,d])<\infty
$$
where the supremum is over all boxes of geodesics $[a,b]\times [c,d]$ with
$\mathcal{L}([a,b]\times [c,d])=\log 2$. The pull backs $h^{*}(\mathcal{L})$
for $h$ quasisymmetric are bounded geodesic currents.

The space of bounded geodesic currents is endowed with the family of H\"older
norms parametrized with the H\"older exponents $0<\nu\leq 1$ (cf.
\cite{Sa2}). The pull backs of the Liouville measure define a homeomorphism
of $T(\mathbb{D})$ onto its image in the bounded geodesic currents; the
homeomorphism is differentiable with a bounded derivative  (cf. \cite{Sa4}) and, in fact, Otal \cite{Ot} proved that the embedding is
real-analytic. The asymptotic rays to the image of $T(\mathbb{D})$  are
identified with the space of projective bounded measured laminations (cf.
\cite{Sa2}). Thus Thurston's boundary of $T(\mathbb{D})$ is the space
$PML_{bdd}(\mathbb{D})$ of all projective bounded measured laminations on
$\mathbb{D}$ (and an analogous statement holds for any hyperbolic Riemann
surface). Alternatively, the space of geodesic currents can be endowed with
the uniform weak* topology (for definition cf. \cite{MiSar}) and Thurston's boundary for
$T(\mathbb{D})$ is again $PML_{bdd}(\mathbb{D})$ (cf. \cite{Sa3}).

\section{The asymptotics of the modulus}\label{section:background}

Let $\G$ be a family of rectfiable curves in $\mathbb{C}$. An {\it admissible metric} $\rho$ for
$\Gamma$ is a non-negative Borel measurable function on
$\mathbb{D}$ such that the $\rho$-length of each
$\gamma\in\Gamma$ is at least one, namely
$$
l_{\rho}(\gamma )=\int_{\gamma}\rho (z)|dz|\geq 1.
$$

The {\it modulus} $\m(\Gamma)$ of the family
$\Gamma$ is given by
$$
\m(\Gamma)=\inf_{\rho}\int_{\mathbb{D}}\rho(z)^2dxdy
$$
where the infimum is over all admissible metrics $\rho$.

We will mostly be interested in estimating moduli of families of curves in a
domain $\Omega\subset\mathbb{C}$ connecting two subsets of the boundary of
$\O$. Thus, given $E,F\subset\partial\O$ we denote $(E,F;\O)$ the family of
rectifiable curves $\g$ having one endpoint  in $E$ and the other endpoint in
$F$. When $\O$ is the unid disc $\mathbb{D}$ and $(a,b,c,d)$ is a quadruple
of distinct points on the boundary circle $\mathbb{S}^1$ given in the
counterclockwise order we denote
$$\G_{[a,b]\times[c,d]} = ((a,b),(c,d);\mathbb{D}).$$
%
%

Lemma \ref{lemma:modproperties} below, summarizes some of the main properties
of the modulus, which we will use repeatedly throughout the paper. We refer
the reader to \cite{GM,LV,Vaisala:lectures} for the proofs of these
properties below and for further background on modulus.

If $\G_1$ and $\G_2$ are curve families in $\mathbb{C}$, we will say that
$\G_1$ \emph{overflows} $\G_2$ and will write $\G_1>\G_2$ if every curve
$\g_1\in \G_1$ contains some curve $\g_2\in \G_2$.

\begin{lemma} \label{lemma:modproperties}
Let $\G_1,\G_2,\ldots$ be curve families in $\mathbb{C}$. Then
\begin{itemize}
  \item[1.] \textsc{Monotonicity:} If $\G_1\subset\G_2$ then $\m(\G_1)\leq
      \m(\G_2)$.
  \item[2.] \textsc{Subadditivity:} $\m(\bigcup_{i=1}^{\infty} \G_i) \leq
      \sum_{i=1}^{\infty}\m(\G_i).$
  \item[3.] \textsc{Overflowing:} If $\G_1<\G_2$ then $\m \G_1 \geq \m
      \G_2$.
\end{itemize}
\end{lemma}

Heuristically modulus of $(E,F;\O )$ measures the amount of curves
connecting $E$ and $F$ in the $\O$. The more ``short" curves there are the
bigger the modulus is. This heuristic may be made precise using a notion of
relative distance $\D(E,F)$, which we define next.

Given two continua $E$ and $F$ in $\mathbb{C}$ we denote
\begin{align}
  \D(E,F) := \frac{\mathrm{dist}(E,F)}{\min\{\diam E, \diam F\}},
\end{align}
i.e. $\D(E,F)$ is the \textit{relative distance} between $E$ and $F$ in
$\mathbb{C}$.

\begin{lemma}\label{lemma:mod-reldist}
For every pair of continua $E,F\subset\mathbb{C}$ we have
\begin{align}\label{modest:reldistance}
\m(E,F;\mathbb{C}) \leq \pi\left(1+\frac{1}{2\D(E,F)}\right)^2.
\end{align}
\end{lemma}
\begin{proof}
  Let $\d:=\dist(E,F)$ and $\G_E^{\d}$ be the family of curves
  $\g\subset\mathbb{C}$ such that $\g(0)\in E$ and $\dist(\g(1),E)\geq \d$.
  Then $(E,F;\mathbb{C})\subset \G_E^{\d}$ and similarly $(E,F;\mathbb{C})\subset
  \G_F^{\d}$. Therefore,
  \begin{align}\label{modest:min}
\m(E,F;\mathbb{C})\leq \min \{\m\G_E^{\d},\m\G_F^{\d}\}.
  \end{align}
Denoting by $E^{\d}$ the $\d$-neighborhood of the set $E\subset\mathbb{C}$,
we note, that
  \begin{align*}
    \rho(z)= \d^{-1}\chi_{E^{\d}}(z)
  \end{align*}
is admissible for $\G^{\d}_E$. Therefore, we have
  \begin{align}\label{modest:reldist1}
\begin{split}\m\G^{\d}_E
&\leq\int_{E^{\d}}(\d^{-1})^2 dxdy = \d^{-2}\mathcal{H}^2(E^{\d})
\leq \d^{-2} \, \pi \left(\frac{\diam E +2\d}{2}\right)^2 \\
&= \pi \left(1+\frac{\diam E}{2\dist(E,F)}\right)^2,
 \end{split} \end{align}
 where $\mathcal{H}^2(E^{\d})$ is the Euclidean area of $E^{\d}$.
Combining inequalities (\ref{modest:min}) and (\ref{modest:reldist1}) we
obtain (\ref{modest:reldistance}).
\end{proof}
\begin{corollary} Let $E_n$ and $F_n$, $n\in\mathbb{N},$ be a sequence of pairs of continua in
$\mathbb{C}$.
  If  the sequence $\D(E_n,F_n)$ is bounded away from $0$ then  $\m(E_n,F_n;\mathbb{C})$ is bounded.
\end{corollary}
\begin{remark}
  The previous lemma is very weak for large $\D(E,F)$, since it is in fact
  easy to see that $\m(E,F,\mathbb{C})$ tends to $0$ as $\D(E,F)\to\infty$.
  But we will not need this estimate in the present paper and will refer the
  interested reader to Heinonen's book \cite{Hein} for relations between the
  modulus and relative distance.
\end{remark}
 The following lemma is an easy consequence of the asymptotic properties
of the moduli (cf. \cite{LV}).

\begin{lemma}[cf. \cite{HaSar}]
\label{lem:mod_liouville_measure} Let $(a,b,c,d)$ be a quadruple of points on
$\mathbb{S}^1$ in the counterclockwise order. Let $\Gamma_{[a,b]\times
[c,d]}$ consist of all differentiable curves $\gamma$ in $\mathbb{D}$ which
connect $[a,b]\subset \mathbb{S}^1$ with $[c,d]\subset \mathbb{S}^1$. Then
$$
\m(\Gamma_{[a,b]\times [c,d]})-\frac{1}{\pi}\mathcal{L}([a,b]\times [c,d])-\frac{2}{\pi}\log 4\to 0
$$
as $\m(\Gamma_{[a,b]\times [c,d]})\to\infty$, where $\mathcal{L}$ is the
Liouville measure.
\end{lemma}

\begin{remark} Note that simultaneously $\m(\Gamma_{[a,b]\times [c,d]})\to\infty$ and $\mathcal{L}([a,b]\times [c,d])\to\infty$.
\end{remark}

\section{The convergence of Teichm\"uller rays}

Let $\varphi$ be an integrable holomorphic quadratic differential on the unit
disk $\mathbb{D}$. In other words, $\varphi :\mathbb{D}\to\mathbb{C}$ is
holomorphic and
$$
\|\varphi\|_{L^1} =\iint_{\mathbb{D}}|\varphi (z)|dxdy<\infty .
$$

A point $z\in\mathbb{D}$ is said to be {\it regular} for $\varphi $ if $\varphi (z)\neq 0$. In a neighborhood of every regular point of $\varphi$ the parameter $w$ given by path integral $w=\int\sqrt{\varphi (z)}dz$ is called a {\it natural parameter} for $\varphi$. The holomorphic quadratic differential $\varphi (z)dz^2$ has representation $dw^2$ in the natural parameter $w$. Moreover, if $w'$ is another natural parameter then $w'=\pm w+const$ at their intersection (cf. \cite{Str}).

A {\it vertical arc} for $\varphi$ is a differentiable arc $\gamma: (a,b)\to\mathbb{D}$ that passes only through regular points of $\varphi$ and that satisfies $\varphi (\gamma (t))\gamma'(t)^2<0$ for all $t\in (a,b)$. Equivalently, a vertical arc is an inverse image of a Euclidean vertical arc in the natural parameter $w$. A {\it vertical trajectory} of $\varphi$ is a {\it maximal} vertical arc. Similarly, a {\it horizontal arc} for $\varphi$ is a differentiable arc $\gamma: (a,b)\to\mathbb{D}$ that passes through regular points and satisfies $\varphi (\gamma (t))\gamma'(t)^2>0$ for all $t\in (a,b)$ (cf. \cite{Str}).

Each end of a vertical trajectory either accumulates to a zero of $\varphi
(z)$ or to the boundary $\mathbb{S}^1$ of $\mathbb{D}$. In particular, if an
end of a vertical trajectory of $\varphi$ accumulates to the boundary
$\mathbb{S}^1$ then the limit set on $\mathbb{S}^1$ consists of a single
point and we say that the vertical trajectory has an endpoint on
$\mathbb{S}^1$ (cf. \cite{Str}). The set of zeroes of $\varphi$ is countable
and therefore only countably many vertical trajectories have an endpoint at a
zero of $\varphi$. Any vertical trajectory of $\varphi$ not in the above
countable set has two distinct endpoints on $\mathbb{S}^1$ (cf. \cite{Str}).

We define the {\it width} of a curve $\gamma$ in $\mathbb{D}$. By Strebel
\cite{Str}, the unit disk $\mathbb{D}$ can be decomposed into countably many
disjoint open strips $S(\beta_i)$ up to a countable family of
vertical trajectories, where $\beta_i$ is an open horizontal
arc and $S(\beta_i )$ is the union of vertical trajectories intersecting $\beta_i$. The strips $S(\beta_i)$ are
open and simply connected. The natural parameter
$$
w=\int\sqrt{\varphi (z)}dz
$$
is well-defined on each $S(\beta_i)$ since $S(\beta_i)$ is simply connected and does
not contain any zeroes of $\varphi$. Any Borel $A\subset \beta_i$ has
well-defined width
$$
width(A)=\int_A|\sqrt{\varphi (z)}dz|.
$$

If $\gamma\subset S(\beta_i)$, denote by $\pi_{\beta_i} (\gamma )$ the
projection of $\gamma$ onto $\beta_i$ along the vertical trajectories. Then
the width of $\gamma$ is defined by
$$
width(\gamma )=\int_{\pi_{\beta_i}(\gamma )}|\sqrt{\varphi (z)}dz|.
$$

Assume that $\gamma$ is not contained in a single strip. Consider the
collection of Borel sets  $\pi_{\beta_i}(\gamma\cap S(\beta_i))$ for all $i$
with $\gamma\cap S(\beta_i)\neq\emptyset$.  We
define the {\it width} of $\gamma$ by
$$
width(\gamma )=\sum_{i=1}^{\infty} width(\gamma\cap S(\beta_i)).
$$

The definition $width(\gamma )$ is given in terms of the strips $S(\beta_i)$.
To see that $width(\gamma )$ is independent of the choice of the strips, let
$S(\beta_j')$ be another countable collection of disjoint open strips that
covers $\mathbb{D}$ up to countable union of singular vertical trajectories.
The two strips $S(\beta_i)$ and $S(\beta_j')$ are either disjoint or they
intersect in an open strip $S(\beta_{i,j})$, where $\beta_{i,j}$ is an open
subinterval on $\beta_i$ which can be homotoped modulo vertical trajectories
to subinterval of $\beta_j'$. The homotopy is measure preserving for $
\int_{*}|\sqrt{\varphi (z)}dz|$. Since $\beta_i-\cup_j\beta_{i,j}$ is at
most countable (which is of measure zero), it follows that
$$
width(\gamma\cap S(\beta_i) )=\sum_j width(\gamma\cap S(\beta_{i,j})).
$$
This implies that $width(\gamma )$ is independent of the choice of the
covering by the strips.

\begin{proposition}
\label{prop:upper_bound} Let $\Gamma =\Gamma ([a,b]\times [c,d])$ be the
family of rectifiable arcs in $\mathbb{D}$ with one endpoint in $[a,b]\subset
\mathbb{S}^1$ and the other endpoint in $[c,d]\subset \mathbb{S}^1$. Denote
by $T_{\epsilon}$ the Teichm\"uller map of $\mathbb{D}$ that shrinks the
vertical trajectories of $\varphi$ by the multiplicative constant $\epsilon
>0$. Then
$$
\limsup_{\epsilon\to 0^{+}}\epsilon\cdot \m(T_{\epsilon}(\Gamma ))\leq \m(\Gamma_v([a,b],[c,d]))
$$
where $\Gamma_v([a,b],[c,d])$ is the set of vertical trajectories of $\varphi$ with one
endpoint in $[a,b]$ and the other endpoint in $[c,d]$.
\end{proposition}

\begin{proof}
By Strebel \cite{Str}, almost every point of $\mathbb{S}^1$ is at a finite
distance from an interior point of $\mathbb{D}$ in the path metric
$\int_{*}\sqrt{|\varphi (z)}dz|$, called $\varphi${\it -metric}. Let
$a',b',c',d'\in \mathbb{S}^1$ be on finite distances from an interior point
such that $[a,b]\subset [a',b']$ and $[c,d]\subset [c',d']$. Let
$\Gamma'=\Gamma ([a',b']\times [c',d'])$.

\begin{figure}
\centering
\includegraphics[width=6cm]{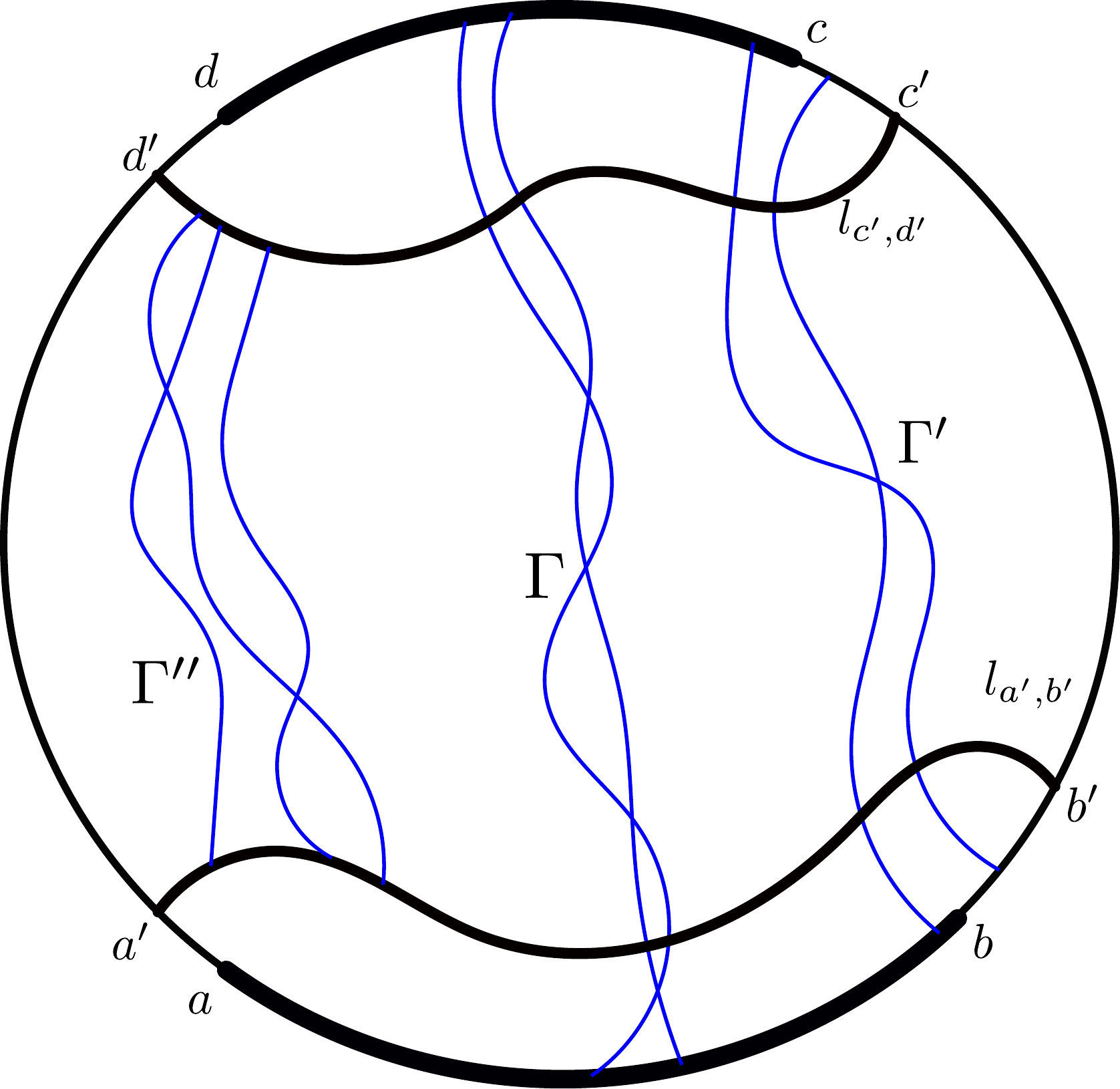}
\caption{The curve families $\Gamma$, $\Gamma'$ and $\Gamma''$.}
\end{figure}

Namely, let (cf. Figure 2) $$\Gamma'=\{\gamma | \,\gamma\mbox{ is rectifiable
and has endpoints in }[a',b']\mbox{ and }[c',d']\}.$$ Since
$\Gamma\subset\Gamma'$, we have $\m(T_{\epsilon}(\Gamma ))\leq
\m(T_{\epsilon}(\Gamma'))$. Let $l_{a',b'}$ and $l_{c',d'}$ be two simple
non-intersecting differentiable arcs in $\mathbb{D}$ with endpoints $a',b'$
and $c',d'$, respectively. Let $\mathbb{D}'$ be the subset of $\mathbb{D}$
with boundary consisting of arcs $l_{a',b'}$, $[b',c']\subset \mathbb{S}^1$,
$l_{c',d'}$ and $[d',a']\subset \mathbb{S}^1$.  Let $\Gamma''=\Gamma
(l_{a',b'}\times l_{c',d'})$ be the family of rectifiable curves in
$\mathbb{D}'$ that connect $l_{a',b'}$ and $l_{c',d'}$. Then the family
$\Gamma'$ overflows the family $\Gamma''$ and we have
\begin{equation}
\label{eq:G'_G''}
\m(T_{\epsilon}(\Gamma'))\leq \m(T_{\epsilon}(\Gamma'')).
\end{equation}

Fix $\eta >0$ and define
$$
\Gamma_{>\eta}''=\{ \gamma\in\Gamma'' |\, width(\gamma )>\eta\}
$$
and
$$
\Gamma_{\leq\eta}''=\{ \gamma\in\Gamma'' |\, width(\gamma )\leq\eta\}.
$$
By the subadditivity of the modulus
$$
\m(T_{\epsilon}(\Gamma''))\leq \m(T_{\epsilon}(\Gamma''_{>\eta}))+ \m(T_{\epsilon}(\Gamma''_{\leq\eta})).
$$
First consider $\m(T_{\epsilon}(\Gamma''_{>\eta}))$. Define the metric
$\rho_{\epsilon}(w)=\frac{1}{\eta}|\sqrt{\varphi_{\epsilon} (w)dw^2}|$ for
$w\in\mathbb{D}'_{\epsilon}$, where $\varphi_{\epsilon}$ is the terminal
holomorphic quadratic differential on
$T_{\epsilon}(\mathbb{D}')=\mathbb{D}_{\epsilon}'$ (cf. \cite{GL}). Recall that the terminal
quadratic differential on $T_{\epsilon}(\mathbb{D}')$ is obtained as follows.
Let $\zeta$ be the natural parameter of $\varphi$ on $\mathbb{D}'$, i.e.
$d\zeta^2=\varphi (z)dz^2$; let $\omega =T_{\epsilon,\zeta }(\zeta )$, where
$T_{\epsilon,\zeta }$ shrinks the vertical direction of $\zeta$ by the
multiplicative constant $\epsilon >0$. Then the terminal quadratic differential $\varphi_{\epsilon}$ is defined in the image of the natural parameter as $\varphi_{\epsilon}(\omega
)d\omega^2=d\omega^2$. If $w=T_{\epsilon}(z)$ then
$\varphi_{\epsilon}(w)dw^2=d\omega^2$.

The metric $\rho_{\epsilon}$ is admissible for
$T_{\epsilon}(\Gamma''_{>\eta})$ since $width(T_{\epsilon}(\gamma ))>\eta$
for all $\epsilon >0$ and all $\gamma \in T_{\epsilon}(\Gamma''_{>\eta})$.
Then
$$
\m(T_{\epsilon}(\Gamma''_{>\eta}))\leq \iint_{T_{\epsilon}(\mathbb{D}')} \rho_{\epsilon}(w)^2dA=\frac{\epsilon}{\eta^2}\iint_{\mathbb{D}'}|\varphi (w)|dA
$$
which gives
\begin{equation}
\label{eq:>eta}
\limsup_{\epsilon\to 0^{+}} \epsilon\cdot \m(T_{\epsilon}(\Gamma''_{>\eta}))=0,
\end{equation}
since $\phi$ is integrable.

We estimate $\m(T_{\epsilon}(\Gamma_{\leq\eta}''))$. Let $z_0\in \mathbb{D}'$
be fixed. Denote by $d^{\varphi}$ the path metric defined by integrating
$|\sqrt{\varphi (z)dz^2}|$ i.e. the $\varphi$-metric. Let $d_0=\max_{z\in l_{a',b'}\cup l_{c',d'}}
d^{\varphi}(z_0,z)$. For $R>0$ define $\mathbb{D}'_R=\{
z\in\mathbb{D}'|d^{\varphi}(z_0,z)\leq R\}$. Given $\epsilon_1>0$ there
exists $R>2d_0$ such that
$$
\iint_{\mathbb{D}'-\mathbb{D}'_R}|\varphi (z)|dA<\epsilon_1.
$$
Denote by $\Gamma_v(l_{a',b'},l_{c',d'})$ the set of vertical trajectories
$\gamma$ connecting $l_{a',b'}$ with $l_{c',d'}$. The choice $R>2d_0$ and the
fact that the vertical trajectories are geodesics for $d^{\varphi}$ implies
that $\Gamma_v(l_{a',b'},l_{c',d'})\subset \mathbb{D}'_R$. From now on we
choose $R=R(\epsilon_1)$ as above.

For $M>0$, define $(\Gamma_{\leq\eta}'')_M=\{ \gamma\in
\Gamma_{\leq\eta}''|\gamma\subset \mathbb{D}'_M\}$. Note that
$$
\Gamma_{\leq\eta}''=(\Gamma_{\leq\eta}'')_{R+1}\cup [ \Gamma_{\leq\eta}''\setminus(\Gamma_{\leq\eta}'')_{R+1}]
$$
which gives
$$
\m(T_{\epsilon}(\Gamma_{\leq\eta}''))\leq \m(T_{\epsilon}((\Gamma_{\leq\eta}'')_{R+1}))+\m (T_{\epsilon}(\Gamma_{\leq\eta}''\setminus(\Gamma_{\leq\eta}'')_{R+1})).
$$

Since $T_{\epsilon}$ is $\epsilon^{-1}$-quasiconformal, we have
\begin{equation*}
\begin{split}
\epsilon\cdot \m (T_{\epsilon}(\Gamma_{\leq\eta}''\setminus(\Gamma_{\leq\eta}'')_{R+1}))\leq \epsilon\cdot \epsilon^{-1}\cdot \m (\Gamma_{\leq\eta}''\setminus(\Gamma_{\leq\eta}'')_{R+1})
=
\m (\Gamma_{\leq\eta}''\setminus(\Gamma_{\leq\eta}'')_{R+1}).
\end{split}
\end{equation*}
Define metric $\rho (z)=\sqrt{|\varphi (z)dz^2|}$ for
$z\in\mathbb{D}'-\mathbb{D}'_R$ and $\rho (z)=0$ otherwise. Then $\rho (z)$
is admissible for the family
$\Gamma_{\leq\eta}''\setminus(\Gamma_{\leq\eta}'')_{R+1}$. Thus
\begin{equation}
\label{eq:>R+1}
\limsup_{\epsilon\to 0^{+}}\epsilon\cdot \m(T_{\epsilon}(\Gamma_{\leq\eta}''-(\Gamma_{\leq\eta}'')_{R+1}))\leq\iint _{\mathbb{D}'-\mathbb{D}'_{R}}|\varphi (z)|dA<\epsilon_1.
\end{equation}

We estimate $\m(T_{\epsilon}((\Gamma_{\leq\eta}'')_{R+1})$. Note that
$\mathbb{D}'_{R+1}$ is a compact metric space for the distance $d^{\varphi}$.
Similar to the above
$$
\epsilon\cdot \m(T_{\epsilon}((\Gamma_{\leq\eta}'')_{R+1}))\leq \m((\Gamma_{\leq\eta}'')_{R+1}).
$$
By Keith \cite{Kei}, we have that
\begin{align}
\limsup_{\eta\to 0^{+}}\m((\Gamma_{\leq\eta}'')_{R+1})\leq
\m(\limsup_{\eta\to 0^{+}}(\Gamma_{\leq\eta}'')_{R+1}).
\end{align}
Recall that a sequence $\{\Gamma_n\}$ of families of curves converges to a family $\Gamma$ if for each $\gamma\in \Gamma$ there exists a subsequence $\gamma_{n_k}\in\Gamma_{n_k}$ such that  uniformly Lipschitz parameterizations of $\gamma_{n_k}$ converge to $\gamma$ as functions when $n_k\to\infty$, and if the limit of each convergent subsequence is a curve in $\Gamma$ (cf. \cite{Kei}).

We establish that
\begin{equation}
\label{eq:lim_curves}
\limsup_{\eta\to 0^{+}}(\Gamma_{\leq\eta}'')_{R+1}=\Gamma_v(l_{a',b'},l_{c',d'}).
\end{equation}
Let $\gamma_n:I\to\mathbb{D}'_{R+1}$ be a sequence of uniformly Lipschitz
parametrizations of curves in $(\Gamma_{\leq{\eta_n}}'')_{R+1}$ with
$\eta_n\to 0$ as $n\to\infty$ that converges to $\gamma :I\to
\mathbb{D}'_{R+1}$. Then
$$
width(\gamma )=0.
$$
Indeed, $width(\gamma )=c>0$ implies that $width(\gamma_n)>c/2>0$ for all $n$
large enough. This contradicts $\gamma_n\in
(\Gamma_{\leq{\eta_n}}'')_{R+1}$.

Since $width(\gamma )=0$, this implies $\gamma \in
\Gamma_v(l_{a',b'},l_{c',d'})$ or that $\gamma$ is composed of several vertical trajectories that meet at a zero of $\varphi$. The later curves are at most countable and their modulus is zero, so we can ignore them. Since $\Gamma_v(l_{a',b'},l_{c',d'})\subset
(\Gamma_{\leq\eta}'')_{R+1}$ by our choice of $R>0$, we obtain
(\ref{eq:lim_curves}). Then (\ref{eq:G'_G''}), (\ref{eq:>eta}),
(\ref{eq:>R+1}) and (\ref{eq:lim_curves}) imply that
\begin{equation}
\label{eq:G_va'}
\limsup_{\epsilon\to 0^{+}}\epsilon\cdot \m(T_{\epsilon}(\Gamma ))\leq \m(\Gamma_v(l_{a',b'},l_{c',d'})).
\end{equation}

We prove that $\Gamma_v(l_{a',b'},l_{c',d'})$ can be replaced by
$\Gamma_v([a,b]\times [c,d])$ in (\ref{eq:G_va'}). Note that (\ref{eq:G_va'})
is true for all $l_{a',b'}$ and $l_{c',d'}$. Choose a sequence $l_{a',b'}^k$
and $l_{c',d'}^k$  such that $l_{a',b'}^k\to [a',b']\subset \mathbb{S}^1$ and
$l_{c',d'}^k\to [c',d']\subset \mathbb{S}^1$ as $k\to\infty$ in the Hausdorff
topology on closed subsets of $\bar{\mathbb{D}}=\mathbb{D}\cup \mathbb{S}^1$.
Denote by $\mathbb{D}_k'$ the subset of $\mathbb{D}$ corresponding to
$l_{a',b'}^k$ and $l_{c',d'}^k$. Define
$$\Gamma_v^k([a',b'],[c',d']):=\Gamma_v([a',b'],[c',d'])\cap \mathbb{D}_k'.
$$

We claim that
\begin{equation}
\label{eq:conv_l_a'}
\lim_{k\to\infty}
\m(\Gamma_v(l_{a',b'}^k,l_{c',d'}^k)-\Gamma_v^k([a',b'],[c',d']))=0.
\end{equation}
Indeed, let $C >0$ be the lower bound on the distance $d^{\varphi}$ between
$l_{a',b'}^k$ and $l_{c',d'}^k$ over all $k$. Then $\rho
(z)=\frac{1}{C}\sqrt{|\varphi (z)|} |dz|$ is admissible for
$\Gamma_v(l_{a',b'}^k,l_{c',d'}^k)$. Let $A_k$ be the union of the (complete)
vertical trajectories in $\mathbb{D}$ that connect $l_{a',b'}^k$ and
$l_{c',d'}^k$ and do not connect $[a',b']$ and $[c',d']$.  Then $A_k\supset
A_{k+1}$ for all $k$ (since we can choose $l_{a',b'}^k$ and $l_{c',d'}^k$
such that $\mathbb{D}'_k\subset\mathbb{D}'_{k+1}$).

We claim that $\cap_{k=1}^{\infty}A_k=\emptyset$. Assume that a horizontal
trajectory $\gamma$ belongs to the union that makes $A_k$. Then there exists
either a Euclidean neighborhood of $[a',b']$ or a Euclidean neighborhood of
$[c',d']$ in $\bar{\mathbb{D}}=\mathbb{D}\cup \mathbb{S}^1$ such that
$\gamma$ is disjoint from this neighborhood. There exists $k'>k$ such that
$\gamma$ does not intersect either $l_{a',b'}^{k'}$ or $l_{c',d'}^{k'}$. Thus
$\gamma$ does not belong to $\cap_{k=1}^{\infty}A_k$ and
$\cap_{k=1}^{\infty}A_k=\emptyset$. This gives
$$
\iint_{A_k}|\varphi (z)|dxdy\to 0
$$
as $k\to\infty$ and (\ref{eq:conv_l_a'}) follows. From (\ref{eq:G_va'}) and
(\ref{eq:conv_l_a'}) we get
\begin{equation}
\label{eq:vert_a'}
\limsup_{\epsilon\to 0^{+}}\epsilon\cdot \m(T_{\epsilon}(\Gamma ))\leq \lim_{k\to\infty}\m(\Gamma_v^k([a',b'],[c',d'])).
\end{equation}

By Keith \cite{Kei}, we have that
$$
\lim_{k\to\infty}\m(\Gamma_v^k([a',b'],[c',d']))\leq \m( \limsup_{k\to\infty} \Gamma_v^k([a',b'],[c',d']))
$$
where $\limsup_{k\to\infty} \Gamma_v^k([a',b'],[c',d'])$ is for the Euclidean
metric on $\bar{\mathbb{D}}=\mathbb{D}\cup \mathbb{S}^1$. For every point of
$\gamma$ which is not a zero of $\varphi$, there exists an open subarc of
$\gamma$ containing the point that is a part of a vertical trajectory of
$\varphi$ because $\gamma_k$ are vertical trajectories. Moreover the limit
$\gamma$ has one endpoint in $[a',b']$ and the other endpoint in $[c',d']$
because $\gamma_k$ has one endpoint on $l_{a',b'}^k$ and one endpoint on
$l_{c',d'}^k$, and $l_{a',b'}^k$  converges to $[a',b']$ and $l_{c',d'}^k$
converges to $[c',d']$.  Therefore every limit $\gamma$ is a vertical
trajectory that necessarily belongs to $\Gamma_v([a',b'],[c',d'])$ or it is
composed of several vertical trajectories meeting at zeros of $\varphi$. The
later family is countable and of zero modulus and without loss of generality
we ignore it. Therefore
\begin{equation}
\label{eq:ineq_vert}
\limsup_{\epsilon\to 0^{+}}\epsilon\cdot \m(T_{\epsilon}(\Gamma )\leq \m(\Gamma_v([a',b'],[c',d'])).
\end{equation}

We choose sequences $[a_k',b_k']\supset [a,b]$ and $[c_k',d_k']\supset [c,d]$
on finite distance from $z_0$ such that $a'_k\to a$, $b_k'\to b$, $c_k'\to c$
and $d_k'\to d$ as $k\to\infty$. The inequality (\ref{eq:ineq_vert}) holds
for these sequences and we need to prove that it holds for
$\Gamma_v([a,b],[c,d])$ as well. It is enough to prove that
$$
\m(\Gamma_v([a_k',b_k'],[c_k',d_k'])-\Gamma_v([a,b],[c,d]))\to 0.
$$

Let $\mathbb{D}_k$ be the union of vertical trajectories in
$\Gamma_v([a_k',b_k'],[c_k',d_k'])-\Gamma_v([a,b],[c,d])$ and note
$\cap_{k=1}^{\infty}\mathbb{D}_k=\emptyset$. Define $\rho
(z)=1/l_v(z)\sqrt{|\varphi (z)dz^2|}$ for $z\in\mathbb{D}_k$ and $\rho (z)=0$
otherwise, where $l_v(z)$ is the length of the vertical trajectory through
$z$ with respect to the metric $d^{\varphi}$. Then $\rho$ is allowable metric
for the family $\Gamma_v([a_k',b_k'],[c_k',d_k'])-\Gamma_v([a,b],[c,d])$.
We have
$$
\m(\Gamma_v([a_k',b_k'],[c_k',d_k'])-\Gamma_v([a,b],[c,d]))\leq\iint_{\mathbb{D}_k}\frac{1}{l_v(z)^2}|\varphi (z)|dxdy.
$$
We claim that $l_v(z)$ has a positive lower bound in $\mathbb{D}_k$. Indeed,
since intervals $[a_k',b_k']$ and $[c_k',d_k']$ are disjoint and decreasing,
their distance in $d^{\varphi}$ metric is positive which implies that any
vertical trajectories connecting them must have lengths bounded below by a
positive constant. Thus $\frac{1}{l_v(z)^2}$ is bounded above. Then
$\cap_{k=1}^{\infty} \mathbb{D}_k=\emptyset$ implies that
$\iint_{\mathbb{D}_k}\frac{1}{l_v(z)^2}|\varphi (z)|dxdy\to 0$ as
$k\to\infty$. The proof is finished.
\end{proof}

\begin{theorem}
\label{thm:main} Let $\Gamma$ be the family of rectifiable arcs in
$\mathbb{D}$ with one endpoint in $[a,b]\subset \mathbb{S}^1$ and the other
endpoint in $[c,d]\subset \mathbb{S}^1$. Denote by $T_{\epsilon}$ the
Teichm\"uller map of $\mathbb{D}$ that shrinks the vertical trajectories of
$\varphi$ by the multiplicative constant $\epsilon >0$. Then
$$
\lim_{\epsilon\to 0^{+}}\epsilon\cdot \m(T_{\epsilon}(\Gamma ))= \m(\Gamma_v([a,b],[c,d]))
$$
where $\Gamma_v([a,b],[c,d])$ is the set of vertical trajectories with one
endpoint in $[a,b]$ and the other endpoint in $[c,d]$.
\end{theorem}

\begin{proof}
We keep the notation as in the proof of Proposition \ref{prop:upper_bound}.
Since $\Gamma_v([a,b],[c,d])\subset \Gamma$, it follows that
$\m(\Gamma_v([a,b],[c,d]))\leq \m(\Gamma )$. Because $\Gamma_v([a,b],[c,d])$
consists of only vertical trajectories, it follows that
$$
\epsilon\cdot \m(T_{\epsilon}(\Gamma_v([a,b],[c,d])))=\m(\Gamma_v([a,b],[c,d])).
$$
Thus
$$
\m(\Gamma_v([a,b],[c,d]))\leq\liminf_{\epsilon\to 0^{+}} \epsilon\cdot \m(T_{\epsilon}(\Gamma )).
$$
The opposite inequality is obtained in Proposition \ref{prop:upper_bound} and
theorem follows.
\end{proof}

\vskip .2 cm

We give an equivalent definition of $\m(\Gamma_v([a,b],[c,d]))$.

\vskip .2 cm

\begin{proposition}
\label{prop:vertical_mod_measure} Let $\varphi$ be an integrable holomorphic
quadratic differential on the unit disk $\mathbb{D}$. Then $$
\m(\Gamma_v([a,b],[c,d]))=\int_I \frac{1}{l(z)}|Re(\sqrt{\varphi (z)}dz)|
$$
where $I$ is at most countable set of horizontal arcs that intersects each trajectory of $\Gamma_v([a,b],[c,d])$ in one point and no other vertical trajectories up to countably many of them, and $l(z)$ is the length of the vertical trajectory through $z$.
\end{proposition}

\begin{proof}
The metric $\rho (z)= \frac{1}{l(z)}|\sqrt{\varphi (z)}dz|$ is
allowable for the family $\Gamma_v([a,b],[c,d])$ and thus
$\m(\Gamma_v([a,b],[c,d]))\leq \int_I \frac{1}{l(z)}|\sqrt{\varphi
(z)}dz|$.

We claim that $\rho (z)$ is extremal metric for the family $\Gamma_v([a,b],[c,d])$  which proves that we have
equality above. Using Beurling's criterion of sufficiency for extremal metrics
\cite{Ahl},  we need to show that if $\int_{\gamma}h_0(z)|dz|\geq 0$ for all
$\gamma\in \Gamma_v([a,b],[c,d])$ and some $h_0:\mathbb{D}\to\mathbb{R}$ then we have
$\iint_{\mathbb{D}}h_0(z)\rho (z)^2dxdy\geq 0$. By transferring the
integration to the natural parameter, we get that $\gamma $ are subsets of
vertical lines which implies $|dz|=dy$ and $\rho (z)=1/l(z)$. Note that $l(z)=l(x)$ is independent of $y$. Then
$\int_{\gamma}h_0(z)|dz|=\int_I h_0(z)dy\geq 0$ and multiplying with $1/l(x)^2$ and an integration in the $x$
direction gives the desired inequality (cf. \cite{HaSar}).
\end{proof}

Define a measured lamination $\mu_{\varphi}$ as follows. The support of
$\mu_{\varphi}$ is a geodesic lamination $v_{\varphi}$ obtained by taking
geodesics in $\mathbb{D}$ which are homotopic to the vertical trajectories of
$\varphi$ relative their endpoints on $\mathbb{S}^1$, i.e. a geodesic in the
support $v_{\varphi}$ of $\mu_{\varphi}$ has endpoints equal to a vertical
trajectory of $\varphi$. For a box of geodesics $[a,b]\times [c,d]$, define
$$
\mu_{\varphi}([a,b]\times [c,d])=\m(\Gamma_v([a,b],[c,d])).
$$
Note that $\mu_{\varphi}$ is a measure on the space of geodesics (i.e. it is
countable additive) by the above integration formula for
$\m(\Gamma_v([a,b],[c,d])).$ Also note that $\m(\Gamma ([a,b],[c,d]))$ is not
countably additive (since moduli are only countably subadditive) and hence it
does not define a measure on $\mathbb{S}^1\times \mathbb{S}^1-diag$.

\vskip .2 cm

\begin{proposition}
\label{prop:common_endpoints} Let $\mu_{\varphi}$ be the measured geodesic lamination
corresponding to an integrable holomorphic quadratic differential $\varphi$
on $\mathbb{D}$ as above. Then
$$
\mu_{\varphi}(\{ a\}\times [c,d])=0
$$
for all $a\in \mathbb{S}^1$ and $[c,d]\subset \mathbb{S}^1$ with $a\notin
[c,d]$.
\end{proposition}

\begin{proof}
We recall that $\mathbb{D}$ is covered by countably many mutually disjoint
open strips $S(\beta_i)$ up to countably many vertical trajectories. Assume
on the contrary that $\mu_{\varphi}(\{ a\}\times [c,d])>0$. Then there exists
an open strip $S(\beta_{i_0})$ such that $$\int_{X_{i_0}}|\sqrt{\varphi
(z)dz^2}|>0,$$ where $\beta_{i_0}$ is the open arc on a horizontal
trajectory and $X_{i_0}=\beta_{i_0}\cap \Gamma_v(\{ a\}\times [c,d])$. By the
definition, $\int_{X_{i_0}}| \sqrt{\varphi (z)dz^2}|$ is the
horizontal measure in the natural parameter of $\varphi$ of the vertical trajectories of $\varphi$ intersecting
$\beta_{i_0}$.

For $z\in X_{i_0}$, let $l(z)$ be the length of the vertical
trajectory through $z$. Since $\varphi$ is integrable, we have that
$$
\int_{X_{i_0}}l(z)|
\sqrt{\varphi (z)dz^2}|<\infty
$$
which implies that $l(z)<\infty$ for a.e. $z\in X_{i_0}$.

Let $z_1,z_2\in X_{i_0}$ be such that there exists $z_1',z_2'\in X_{i_0}$
with $z_1<z_1'<z_2'<z_2$ for a linear order on $\beta_{i_0}$, and $l(z_1')$
and $l(z_2')$ finite. Let $\gamma_{z_i},\gamma_{z_i'}$ be the maximal vertical rays starting
at $z_i,z_i'$ respectively that have $a$ as their common endpoint. Note that
vertical rays $\gamma_{z_1}$ and $\gamma_{z_2}$ do not intersect
$\beta_{i_0'}$ except at their initial points because any two points in
$\mathbb{D}$ can be joined by at most one geodesic arc in the metric
$|\sqrt{\varphi (z)dz^2}|$ (cf. \cite[Theorem 14.2.1, page 72]{Str}). Let $[z_1,z_2]$ be the subarc of the vertical
trajectory between $z_1$ and $z_2$. Then $\gamma_{z_1}\cup\gamma_{z_2}\cup
[z_1,z_2]\cup\{ a\}$ is the boundary of a simply connected domain $U$ inside
$\mathbb{D}$.

We claim that $U$ is a Jordan domain. Indeed, since $\gamma_{z_1}$,
$\gamma_{z_2}$ and $[z_1,z_2]$ are simple geodesic arcs which meet only at
their endpoints, it follows that $\gamma_{z_1}\cup\gamma_{z_2}\cup [z_1,z_2]$
is a Jordan arc. We parametrize it by a homeomorphism $f:\mathbb{S}^1-\{
1\}\to \gamma_{z_1}\cup\gamma_{z_2}\cup [z_1,z_2]$ and extend $f(1)=a$. Then
$f$ is a bijection of $\mathbb{S}^1$ and $\partial
U=\gamma_{z_1}\cup\gamma_{z_2}\cup [z_1,z_2]\cup\{ a\}$. Moreover, $f$ is
continuous at $1$ since $\gamma_{z_1}$ and $\gamma_{z_2}$ accumulate to $a$
and therefore $\partial U$ is a Jordan curve.

For $z\in [z_1,z_2]$, let $\gamma_z$ be the ray of the vertical trajectory
with the initial point $z$ that starts in the direction of $U$. Then
$\gamma_z$ never leaves $U$ because it cannot intersect its boundary except
at $z$.  Moreover, the ray $\gamma_z$ cannot contain critical points of
$\varphi$. Indeed, if it does contain a critical point then there exist two
vertical rays starting at the critical point which make a geodesic and whose
both accumulation points on $\mathbb{S}^1$ are equal to $a$. However, a
geodesic must have two different accumulation points (cf. \cite[Theorem 19.4
and Theorem 19.6]{Str}) which gives a contradiction. Therefore every vertical
trajectory in $U$ is non-critical and its full extension accumulates at $a\in
\mathbb{S}^1$ and intersects $[z_1,z_2]$ in exactly one point. Therefore, $U$
is foliated by $\gamma_z$ for $z\in (z_1,z_2)$.

Consider the conformal mapping from $U$ into $\mathbb{C}$ using the natural
parameter $dw^2=\varphi (z)dz^2$. Since $U$ is simply connected and without
zeroes, the natural parameter is conformal on $U$. Caratheodory's theorem (cf. \cite{Pomm}) gives that $w$ homeomorphically maps
 the boundary $\partial U$ of $U$ onto the prime ends of $w(U)$.

Since $\gamma_{z_1'}$ and
$\gamma_{z_2'}$ have finite lengths, it follows that the endpoints $w_1'$ and $w_2'$ of vertical lines
$w(\gamma_{z_1'})$ and $w(\gamma_{z_2'})$ are different in $\partial w(U)$. The arcs $\gamma_{z_1'}$ and
$\gamma_{z_2'}$
define degenerate prime ends, namely prime ends whose imprints are $w_1'$ and $w_2'$. Therefore the prime ends are different since $w_1'\neq w_2'$ (cf. Figure 3).

\begin{figure}
\noindent\makebox[\textwidth]{
\includegraphics[width=12 cm]{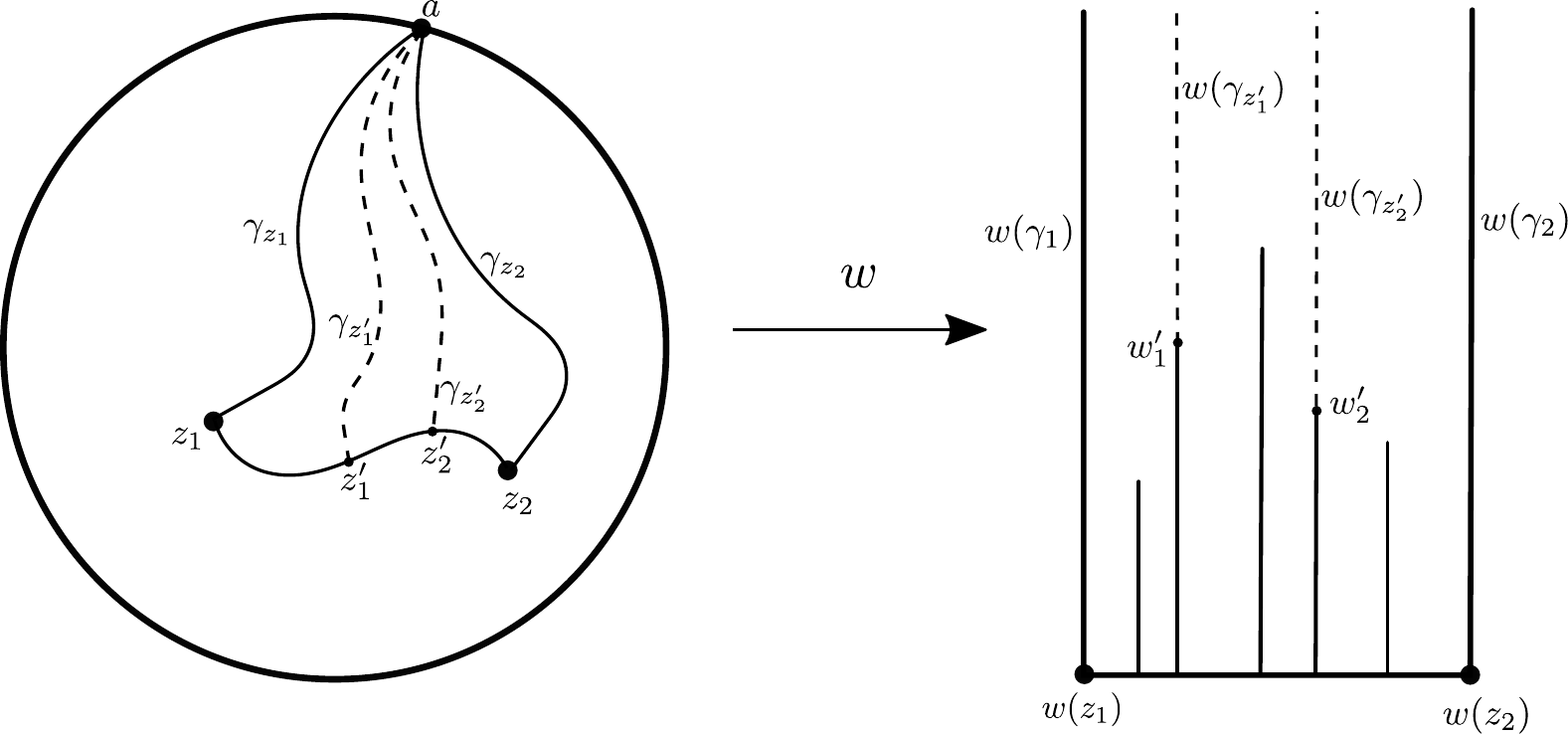}}
\caption{The prime ends of vertical trajectories.}
\end{figure}

This is impossible since $w$ maps $a$ onto
both prime ends.  Contradiction. Thus we
obtained that $\mu_{\varphi}(\{ a\}\times [c,d])=0$.
\end{proof}

Putting the above statements together and using the fact that the asymptotics
of the Liouville currents can be replaced by the asymptotics of the moduli of
curves (cf. Lemma \ref{lem:mod_liouville_measure}) gives

\begin{theorem}
\label{thm:weak*convergence} Let $\varphi$ be an integrable holomorphic
quadratic differential on $\mathbb{D}$ and let $T_{\epsilon}$ be the
Teichm\"uller mapping that shrinks the vertical trajectories of $\varphi$ by
a multiplicative constant $\epsilon >0$. The Teichm\"uller ray
$\epsilon\mapsto T_{\epsilon}$ for $\epsilon >0$ has a unique limit point
$[\mu_{\varphi}]$ on Thurston's boundary $PML_{bdd}(\mathbb{D})$ of
$T(\mathbb{D})$ as $\epsilon \to 0^{+}$, where $[\mu_{\varphi}]$ is the
projective class of a bounded measured lamination $\mu_{\varphi}$
corresponding to $\varphi$.
\end{theorem}

\begin{proof}
The convergence $T_\epsilon\to [\mu_{\varphi}]$ as $\epsilon\to 0^{+}$ in the
weak* topology on measures follows immediately from Theorem \ref{thm:main}, Proposition \ref{prop:common_endpoints} and Lemma \ref{lem:mod_liouville_measure}. It remains to be proved that
$\mu_{\varphi}$ is Thurston bounded.

Note that by the definition the measured lamination $\mu_{\varphi}$ is
independent under multiplication of $\varphi$ by positive constants. Let
$[a,b]\times [c,d]$ be such that its Liouville measure satisfies
$$
\mathcal{L}([a,b]\times [c,d])=\log 2.
$$
Denote by $\Gamma ([a,b],[c,d])$ the family of all rectifiable arcs in
$\mathbb{D}$ that have one endpoint in $[a,b]$ and other endpoint in $[c,d]$.
Then
$$
\m(\Gamma ([a,b],[c,d]))\leq const
$$
for all $\mathcal{L}([a,b]\times [c,d])=\log 2.$ Since $\Gamma_v([a,b],[c,d])\subset
\Gamma ([a,b],[c,d])$, we have that
$$
\mu_{\varphi}([a,b]\times [c,d])=\m(\Gamma_v ([a,b],[c,d]))\leq const
$$
and $\|\mu_{\varphi}\|_{Th}<\infty$.
\end{proof}

\section{A counter-example to uniform weak* convergence}
\label{sec:ce} Let $\{\alpha_n\}_n$ and $\alpha$ be geodesic currents on the
space of geodesics $G(\mathbb{H})=\mathbb{S}^1\times \mathbb{S}^1-diag$ of
the hyperbolic plane $\mathbb{H}$. Namely, $\{\alpha_n\}_n$ and $\alpha$ are
positive Radon measures on $\mathbb{S}^1\times \mathbb{S}^1-diag$. We say
that $\alpha_n$ converges to $\alpha$ in the {\it uniform} weak* topology
(cf. \cite{Sa2}) if for every continuous $f:\mathbb{S}^1\times
\mathbb{S}^1-diag\to\mathbb{R}$ with support in the standard box $[1,i]\times
[-1,-i]$ we have

$$
\sup_{[a,b]\times [c,d]}|\int_{[1,i]\times [-1,-i]}f[d\gamma_{[a,b]\times [c,d]}^{*}(\alpha_n-\alpha )]|\to 0
$$
as $n\to\infty$, where the supremum is over all boxes $[a,b]\times [c,d]$ of Liouville measure $\log 2$ and $\gamma_{[a,b]\times [c,d]}$ is the M\"obius map taking the standard box onto $[a,b]\times [c,d]$.

We note that as $[a,b]\times [c,d]$ runs through all boxes of Liouville
measure $\log 2$, $\gamma_{[a,b]\times [c,d]}$ runs through all M\"obius maps
of $\mathbb{D}$, which implies that the above supremum can be taken over the
space $Mob(\mathbb{D})$ of all M\"obius maps that preserve the unit disk
$\mathbb{D}$. Moreover, since any continuous $f:[a_0,b_0]\times
[c_0,d_0]\to\mathbb{R}$ with $\mathcal{L}([a_0,b_0]\times [c_0,d_0])=\log 2$
can be pulled back to a continuous $f\circ\gamma_{[a_0,b_0]\times
[c_0,d_0]}:[1,i]\times [-1,-i]\to\mathbb{R}$ and since the above supremum is
over all M\"obius maps, we do not need to restrict to continuous functions
with supports on the standard box, but rather to continuous maps with supports
in any box with Liuoville measure $\log 2$. In addition, if a continuous
$f:\mathbb{S}^1\times \mathbb{S}^1-diag\to\mathbb{R}$ has a compact support
then it can be written as a finite sum of continuous functions with supports
in boxes of Liouville measures $\log 2$. Therefore

\begin{definition} A sequence of geodesic currents $\{ \alpha_n\}_n$ converges in the {\it uniform} weak* topology to $\alpha$ if for every continuous function $f:\mathbb{S}^1\times \mathbb{S}^1-diag\to\mathbb{R}$ with compact support
$$
\sup_{\gamma\in Mob(\mathbb{D})}|\int_{\mathbb{S}^1\times \mathbb{S}^1-diag}f[d\gamma^{*}(\alpha_n-\alpha )]|\to 0
$$
as $n\to\infty$.
\end{definition}

This definition is equivalent to the first definition using boxes of Liouville measure $\log 2$.

Assume that $\alpha_n$ converges to $\alpha$ in the weak* topology. Below we formulate a sufficient condition guaranteeing that $\alpha_n$ does not converge to $\alpha$ in the uniform weak* topology. Given $\delta >0$, assume that there exist $C_1,,C_2,C_3$ and a sequence of boxes $Q_k=[a_k,b_k]\times [c_k,d_k]$ and sub-boxes $Q_k'=[a_k',b_k']\times [c_k',d_k']$ compactly contained in the interior of $Q_k$ such that
\begin{equation}
\label{eq:constant-measure}
\mathcal{L}(Q_k)\leq C_1,
\end{equation}

\begin{equation}
\label{eq:bounded-measure}
\mathcal{L}(Q_k')\geq C_2>0,
\end{equation}

\begin{equation}
\label{eq:central}
\begin{split}
\min\{\mathcal{L}([a_k,a_k']\times [c_k,d_k]),\mathcal{L}([b_k',b_k]\times [c_k,d_k]),\ \ \ \ \ \ \ \ \ \\
\mathcal{L} ([a_k,b_k]\times [c_k,c_k']), \mathcal{L}([a_k,b_k]\times [d_k',d_k])\}\geq\delta >0,
\end{split}
\end{equation}

\begin{equation}
\label{eq:lower-bound}
\alpha_{n_k}(Q_k')\geq C_3>0,
\end{equation}
for some $n_k$ with $n_k\to\infty$ as $k\to\infty$, and
\begin{equation}
\label{eq:zero-measure}
\alpha (Q_k)\to 0,
\end{equation}
as $k\to\infty$,
where $C_1'$, $C_1''$, $C_2$ and $C_3$ are independent of $k$ and $\delta$ .

We now establish that $\alpha_n$ does not converge to $\alpha$ in the uniform weak* topology if the above conditions are satisfied. Let $Q=[a,b]\times [c,d]$ be a fixed box with $\mathcal{L}(Q)=C_1$ and let $Q'=[a',b']\times [c',d']$ be a box compactly contained in the interior of $Q$ such that
\begin{equation}
\label{eq:dist-bdry}
\begin{split}
\mathcal{L}([a,a']\times [c,d])=\mathcal{L}([b',b]\times [c,d])=\ \ \ \ \\
\mathcal{L} ([a,b]\times [c',c])=\mathcal{L}([a,b]\times [d',d])=\delta .
\end{split}
\end{equation}
Let $\gamma_k\in Mod(\mathbb{D})$ be such that $\gamma_k(Q)\supseteq Q_{k}$.
Then (\ref{eq:dist-bdry}) and (\ref{eq:central}) imply that
$\gamma_k(Q')\supseteq Q_{k}'$. Let $f:\mathbb{S}^1\times
\mathbb{S}^1-diag\to\mathbb{R}$ be a continuous functions such that the
support of $f$ is contained in $Q$, $0\leq f\leq 1$ and $f|_{Q'}=1$. Then by
(\ref{eq:lower-bound}) and (\ref{eq:zero-measure}) we have
$$
\int_{\mathbb{S}^1\times \mathbb{S}^1-diag}fd\gamma_k^{*}[\alpha_{n_k}-\alpha ]\geq C_3-\alpha (Q_{k})>C_3/2>0
$$
when $n_k$ is large, which implies that $\alpha_{n_k}$ does not converge in the uniform weak* topology to $\alpha$. Since the uniform weak* convergence implies the weak* convergence and since $\alpha_n$ converges in the weak* topology to $\alpha$, it follows that $\alpha_n$ does not converge to any geodesic current in the uniform weak* topology.

We find an example of an integrable holomorphic quadratic differential
$\varphi$ on the unit disk $\mathbb{D}$ such that the corresponding
Teichm\"uller ray $T_{\epsilon}$ does not converge in the uniform weak*
topology to $\mu_{\varphi}$ while Theorem \ref{thm:weak*convergence}
established that it does converge to $\mu_{\varphi}$ in the weak* topology.
The differential $\varphi$ is constructed by taking the pull back of $dz^2$
on the domain $D$ in the lemma below under the Riemann mapping. For
simplicity of notation, we denote by $T_{\epsilon}:\mathbb{S}^1\to
\mathbb{S}^1$ the boundary map of the Teichm\"uller geodesic ray
$T_{\epsilon}$. The above criterion is used for the family of Liouville
currents $\alpha_{\epsilon}:=\epsilon T_{\epsilon}^{*}(\mathcal{L})$ when
$\epsilon\to 0^{+}$ and the weak* limit $\mu_{\varphi}$. The conditions
(\ref{eq:constant-measure}), (\ref{eq:bounded-measure}), (\ref{eq:central}),
(\ref{eq:lower-bound}) and (\ref{eq:zero-measure}) are replaced by equivalent
conditions in terms of the moduli of the families of curves connecting two
intervals on $\mathbb{S}^1$ defining the box of geodesics.

\begin{lemma}\label{lemma:counterexample}
  There is a domain $D\subset\mathbb{C}$ of finite area with the following properties. There exist constants $0<C_1',C_2',C_3',\delta' <\infty$, a sequence of arcs $[a_k,b_k], [c_k,d_k]$ and sub-arcs $[a_k',b_k'],[c_k',d_k']$ on $\mathbb{S}^1$
and a sequence $\eps_k>0$ approaching $0$ such that with notations as above we have
  \begin{itemize}
    \item[(\textit{a}).] $ \m ([a_k,b_k], [c_k,d_k];\mathbb{D})\leq C_1', \forall k\in\mathbb{N}$
    \item[(\textit{b}).] $ \m ([a_k',b_k'], [c_k',d_k'];\mathbb{D})\geq C_2', \forall k\in\mathbb{N}$
     \item[(\textit{c}).] $\min\{\m ([a_k,a_k'], [c_k,d_k];\mathbb{D}),\m ([b_k',b_k], [c_k,d_k];\mathbb{D}), \\
\m ([a_k,b_k], [c_k,c_k'];\mathbb{D}), \m ([a_k,b_k], [d_k',d_k];\mathbb{D})\}\geq\delta' >0,$
    \item[(\textit{d}).] $\mu_{\varphi} ([a_k,b_k]\times [c_k,d_k])\to0, \mbox{ as } k\to\infty$
    \item[(\textit{e}).]
        $\eps_k \cdot \m (T_{\epsilon_k}^{\varphi}([a_k',b_k']), T_{\epsilon_k}^{\varphi}([c_k',d_k']);T_{\epsilon_k}^{\varphi}(\mathbb{D}))\geq C_3'.$
  \end{itemize}
\end{lemma}

\begin{remark}
We would like to emphasize again that in the lemma above $\varphi$ denotes
the quadratic differential which is the pullback of $dz^2$ under the Riemann
map of $D$ and $T_{\epsilon}^{\varphi}$ is the corresponding Teichm\"uller
mapping. The boxes $[a_k,b_k]\times [c_k,d_k]$ and $[a_k',b_k']\times
[c_k',d_k']$ on $\mathbb{S}^1$ under the Riemann mapping correspond to boxes
in $D$, and this correspondence is implicitly assumed.
\end{remark}


\begin{proof}[Proof of Lemma \ref{lemma:counterexample}]

Below we will define the domain $D$ as well as a sequence of continua
$E_k,E_k'$,$F_k,F_k'\subset \partial{D}$, which are the preimages of the
intervals $[a_k,b_k],[a_k',b_k'],[c_k,d_k],[c_k',d_k']\subset \mathbb{S}^1$
under the Riemann mapping of $D$. In particular $E_k'\subset E_k$ and
$F_k'\subset F_k$. Moreover, instead of estimating the moduli of the curve
families in the unit disc $\mathbb{D}$, we will obtain the estimates in $D$.
To simplify the notation we let
\begin{align*}
\G_k:=(E_k,F_k;D)\,\mbox{ and }\,
\G_k':=(E_k',F_k';D).
\end{align*}
Furthermore, denoting the two nonempty components of $E_k\setminus E_k'$ and $F_k\setminus F_k'$ by $E_k^i$ and $F_k^i$, $i\in\{1,2\}$, respectively, we let
\begin{align*}
\G_k^{i,j}:= (E_k^i,F_k^j;D).
\end{align*}
Just as before, given two continua $E,F\subset\partial D$ we denote by $\G_v(E,F;D)$ the family of vertical curves connecting $E$ and $F$ in $D$.

By conformal invariance of the modulus and Theorem \ref{thm:main}  conditions $(a)-(e)$ are equivalent to the following:

  \begin{itemize}
    \item[$(\textit{a}')$.] $ \m \G_k \leq C_1', \forall k\in\mathbb{N},$
    \item[(\textit{b}$'$).] $ \m \G_k'\geq C_2', \forall k\in\mathbb{N},$
     \item[(\textit{c}$'$).] $\m \G_k^{i,j}\geq\delta', \forall k\in\mathbb{N}, \forall i,j\in\{1,2\},$
    \item[(\textit{d}$'$).] $\displaystyle \lim_{k\to\infty} \m \G_v(E_k,F_k;D) = 0,$
    \item[(\textit{e}$'$).]
        $\eps_k \cdot \m (T_{\epsilon_k}(\G_k'))\geq C_3',\forall k\in\mathbb{N}.$
  \end{itemize}

\noindent Next, we construct the domain $D$ and prove properties $(a')-(e')$.

For $k=1,2\ldots,$ and $j=0,1,\ldots,2^k$ let
$$L_{k,j}:=\left\{\left(\frac{1}{2^k}+\frac{j}{2^{2k}},y\right): \frac{1}{2^k}\leq y\leq 1\right\}.$$
Define
$$D:= [0,1]^2 \setminus \bigcup_{k=1}^{\infty}\bigcup_{j=0}^{2^k} L_{k,j}.$$

\begin{figure}
\includegraphics[width=10cm]{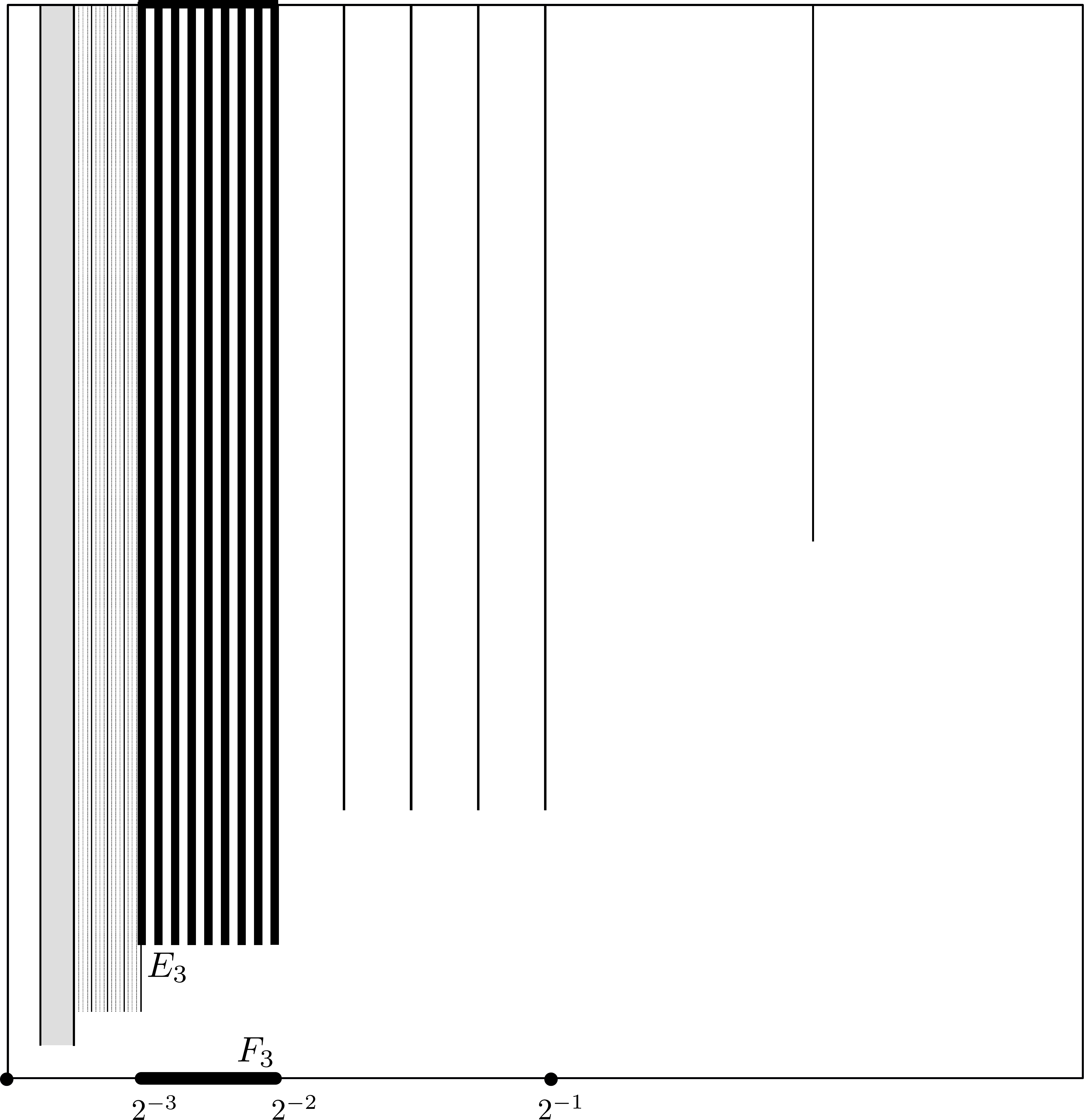}
\caption{The domain $D$. The bold interval is $F_3$, while the part of $\partial{D}$ above it is $E_3$.}
\end{figure}

Now, for $k\geq 1$
let
\begin{align}
F_k=\left[\frac{1}{2^k},\frac{2}{2^k}\right] \quad &\mbox{ and } \quad E_k=\bigcup_{j=0}^{2^k} L_{k,j} \cup \left\{(x,1) : x\in F_k \right\},\\
F_k'=\frac{1}{2}F_k \quad &\mbox{ and } \quad E_k'=\bigcup_{j=\frac{1}{4}2^k}^{\frac{3}{4}2^k} L_{k,j} \cup \left\{(x,1) : x\in F_k' \right\},
\end{align}
where $\frac{1}{2}F_k$ denotes the interval with the same center as $F_k$ but half the length.

\textbf{Proof of (\textit{a}$'$).} Since
$$\D(E_k,F_k)=\frac{\dist(E_k,F_k)}{\min\{\diam E_k,\diam F_k \}} =\frac{2^{-k}}{2^{-k}}=1$$ for every $k\geq 1$, by Lemma \ref{lemma:mod-reldist} we have
$$\m(\G_k)=\m(E_k,F_k;D)\leq \m(E_k,F_k;\mathbb{C})\leq \frac{9}{4}\pi.$$

\textbf{Proof of (\textit{b}$'$).}
To estimate $\m\G_k'$ from below we will use conjugate families. Recall that if continua $E,F\subset \partial D$ then the family of curves separating $E$ and $F$ in $D$ is called the family  conjugate to $(E,F;D)$. We will denote by $(E,F;D)^t$ the family conjugate to $(E,F,D)$. The modulus of $(E,F;D)^t$ may be found as follows, see ~\cite{GM}
\begin{align}
\m ((E,F;D)^t) = \frac{1}{\m(E,F;D)}.
\end{align}

\noindent Thus, to estimate $\m\G_k'$ from below we can instead estimate $\m((\G_k')^t)$ from above. Note, that
every curve $\g\in(\G_k')^t$ contains a subcurve $\d$ connecting the two components of $\partial D \setminus (F_k'\cup E_k')$ in the rectangle $F_k'\times[0,1]$, see Fig. \ref{fig:zoom}. Therefore,
\begin{align}
\m (\G_k')^t \leq \m G_k,
\end{align}
where by $G_k$ we denote the family of curves connecting the two components of $\partial (F_k'\times (0,1)) \setminus (F_k'\cup E_k')$ in the rectangle $F_k'\times[0,1]$.  Next we estimate $\m G_k$ using the following result.

\begin{figure}\label{fig:zoom}
\includegraphics[width=10cm]{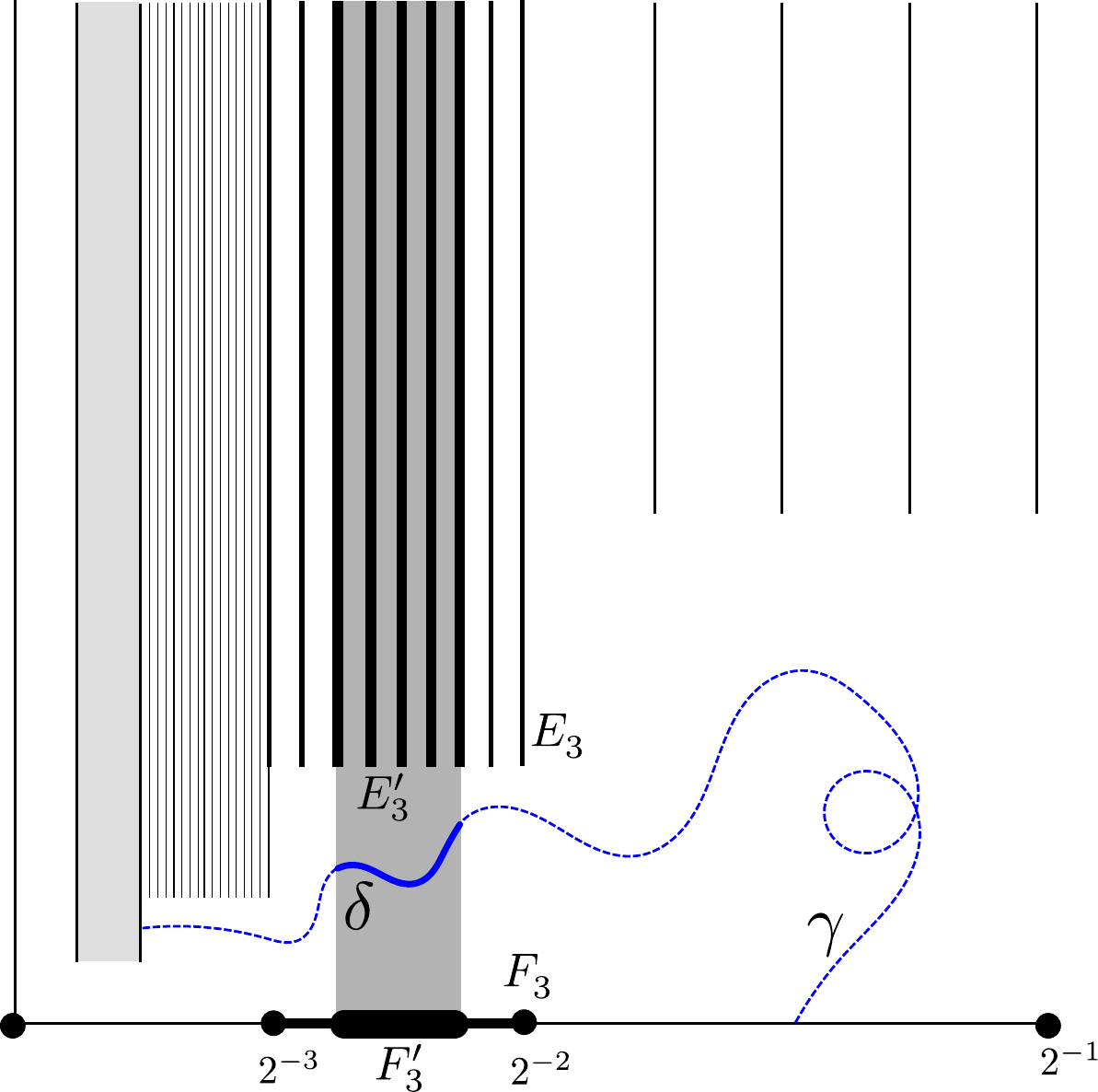}
\caption{Estimating $\m(F_k',E_k';D)$ from below. Every curve $\g\in(\G_k')^t$ separating $E_k'$ from $F_k'$ in $D$ contains a subcurve $\d\in G_k$ which connects the two components of $\partial(F_k'\times(0,1))\setminus (F_k' \cup E_k')$ within  the grey rectangle $F_k'\times(0,1)$.}
\end{figure}

\begin{lemma}\label{lemma:slits}
Let $0<a<b<0$, $0<c<b$ and $N\geq 1$. Denote by $D_N$ the domain
\begin{align*}
D_N=(0,a)\times(0,b)\setminus \bigcup_{i=0}^N \{x_i\} \times [c,1],
\end{align*}
where $x_i=\frac{ai}{N}$ and by $\G_N$ the family of curves in $D_N$ connecting the vertical intervals $(0,ic)$ to $(a,a+ic)$ (see Figure \ref{fig:vertslits}). Then
\begin{align}
\frac{c}{a}\leq\m\G_N\leq \frac{c}{a}+\frac{1}{2N}.
\end{align}
\end{lemma}

\begin{proof}
The first estimate follows from the fact that $\G_N$ contains the family connecting the vertical sides in the rectangle $[0,a]\times [0,c]$. To obtain the upper bound consider the rectangle $R_N=(0,a)\times(0,c+\frac{a}{2N})$ and define
\begin{align*}
\rho_N = \frac{1}{a}\chi_{R_N\cap D_N}.
\end{align*}
We next show that $\rho_N$ is admissible for $\G_N$. For that, let $\g\in\G_N$ and let (cf. Figure 6)
\begin{align*}
\g_i = \g\cap ((x_i,x_{i+1})\times(0,1)), \quad i=0,\ldots,N-1.
\end{align*}
Next, we show that $l(\g_i\cap R_N)\geq x_{i+1}-x_i = \frac{a}{N}.$
%
%
Indeed, if $\g_i\cap \partial R_N = \emptyset$ then, since   $\g_i\cap R_N$  is a connected curve connecting the vertical sides of the rectangle $R_N \cap ((x_i,x_i+1)\times(0,1))$, we have $l(\g_i\cap R_N)\geq a/N$.
On the other hand if $\g_i\cap \partial R_N \neq \emptyset$ then there are two connected components $\g_i',\g_i''$ of $\g_i$ which connect the vertical intervals $\{x_{i}\}\times (0,c)$ and $\{x_{i
+1}\}\times(0,c)$ to the horizontal interval $(x_{i},x_{i+1})\times\{c+\frac{a}{(2N)}\}$ in $D_N$, respectively. Since the distance between the aforementioned vertical intervals and the horizontal interval is at least $a/(2N)$ we obtain
\begin{align*}
l(\g_i\cap R_N)\geq l(\g_i')+l(\g_i'')\geq 2\frac{a}{2N} = \frac{a}{N}.
\end{align*}
\begin{figure}\label{fig:vertslits}
\includegraphics[width=10cm]{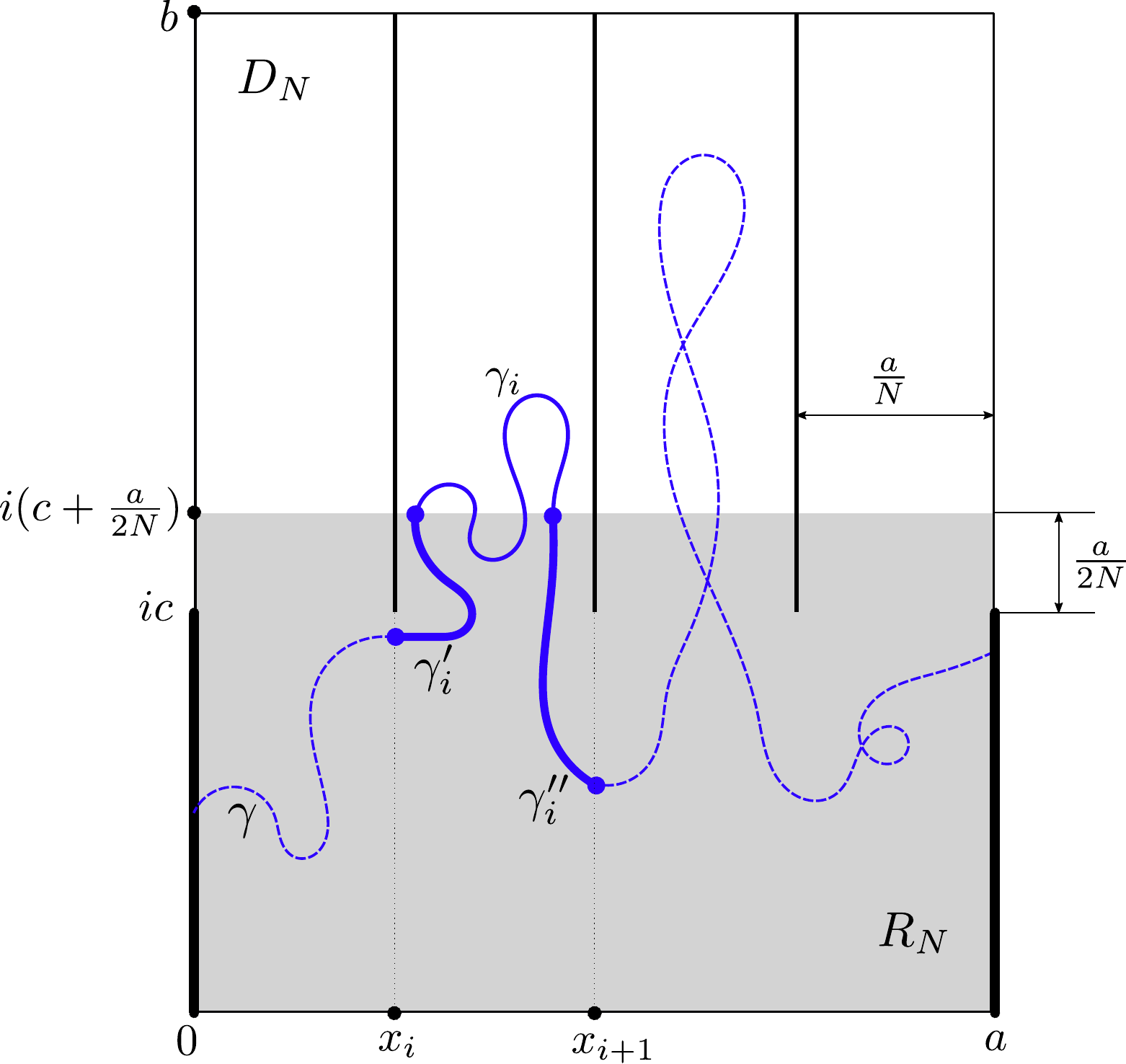}
\caption{}
\end{figure}
Thus we have
\begin{align*}
l_{\rho_N}(\g) \geq \sum_{i=0}^{N-1} l_{\rho_N}(\g_i)= \frac{1}{a} \sum_{i=0}^{N-1} l(\g_i\cap R_N)\geq \frac{1}{a}\cdot N \cdot \frac{a}{N}=1,
\end{align*}
and $\rho_N$ is admissible for $\G_N$. Therefore we can estimate the modulus of $\G_N$ as follows.
\begin{align*}
\m\G_N\leq\int \rho_N^2  = \frac{1}{a^2} |R_N\cap D_N| =\frac{1}{a^2}\cdot a (c+\frac{a}{2N}) = \frac{c}{a} + \frac{1}{2N}. \qquad \qquad \qedhere
\end{align*}
\end{proof}

Using the lemma we see that
\begin{align*}
\m \G_k' \geq \frac{1}{\m((\G_k')^t)} \geq \frac{1}{\m G_k} \geq \frac{1}{\frac{2^{-k}}{2^{-k}/2}+\frac{1}{2\cdot 2^k}} >\frac{1}{3},
\end{align*}
for $k\geq 1$. Which proves $(b')$.

\textbf{Proof of (\textit{c}$'$).}
We start by estimating the modulus of  $\G_k^{1,1}=(E_k^1,F_k^1;D)$. Note that \begin{align*}
F_k^1& = \left(\frac{1}{2^{k}},\frac{1+1/4}{2^k}\right)\\
F_k^2&= \left(\frac{1+3/4}{2^{k}}, \frac{2}{2^{k}}\right),
\end{align*}
while $E_k^1$ and $E_k^2$ are the parts of $E_k$ above $F_k^1$ and $F_k^2$, respectively.

Just like in the proof of $(b')$ we use Lemma \ref{lemma:slits} to obtain the following estimate
\begin{align}
\begin{split}\m(F_k^1,E_k^1,D)
&\geq \m(F_k^1,E_k^1,F_k^1\times[0,1]) \\
&= \frac{1}{\m(F_k^1,E_k^1,F_k^1\times[0,1])^t}
\geq \frac{1}{\frac{2^{-k}}{2^{-k}/4} + \frac{1}{2 (2^k/4)}} \geq \frac{1}{5},
\end{split}
\end{align}
for $k\geq 1$. The same way we also obtain $\m(F_k^2,E_k^2,D) \geq 1/5$.
Next, we estimate $\m(F_k^1,E_k^2,D)$ as follows
\begin{align}
\begin{split}
\m(F_k^1,E_k^2,D)
\geq \m(F_k^1,E_k^2,F_k\times[0,1])= \frac{1}{\m(F_k^1,E_k^2,F_k^1\times[0,1])^t}.
\end{split}
\end{align}
But
\begin{align*}
(F_k^1,E_k^2,F_k^1\times[0,1])^t = (F_k'',E_k'',F_k^1\times[0,1])
\end{align*}
where $E_k''$ and $F_k''$ are the two components of $\partial (F_k\times(0,1))\setminus (F_k \cup E_k)$, or
\begin{align*}
E_k'' &= [2^{-k},(1+i)2^{-k}]\cup E_k^1\cup E_k',\\
F_k'' &= [2^{-k+1},(1+i)2^{-k+1}]\cup F_k'\cup F_k^2.
\end{align*}
Since $\dist(E_k'',F_k'')=\diam F_k^1 = 2^k/4$ and $\diam E_k''\geq\diam F_k''$ we have
\begin{align*}
\D(E_k'',F_k'') = \frac{\diam F_k^1}{\diam F_k''} \geq \frac{2^{-k}/4}{2\cdot 2^{-k}} =\frac{1}{8}.
\end{align*}
Therefore by Lemma \ref{lemma:mod-reldist} we have
\begin{align*}
\m(F_k'',E_k'',F_k^1\times[0,1])\leq \pi(1+4)^2 = 25\pi,
\end{align*}
and we finally obtain
\begin{align*}
\m(F_k^1,E_k^2,D) \geq \frac{1}{25\pi}.
\end{align*}
In the same way we can show that $\m(F_k^2,E_k^1,D))\geq \frac{1}{25\pi}$ and thus
prove $(c')$.

\textbf{Proof of (\textit{d}$'$).} As was shown before we have,
\begin{align*}
\m \G_v (E_k,F_k;D) = \int_{F_k} \frac{dx}{l(x)}
\end{align*}
where in this case $l(x)$ is the Euclidean length of the vertical trajectory passing through $x\in \mathbb{C}$ and integration is with respect to the Lebesgue measure. Thus, since $l(x)=1$ for almost every $x\in F_k$ we obtain
\begin{align*}
\m \G_v (E_k,F_k;D )= |F_k| = \frac{1}{2^k}\to0.
\end{align*}

\textbf{Proof of (\textit{e}$'$).}
Just like above, let $G_k$ be the family of curves connecting the two vertical intervals in $\partial D$ (namely, the two components of the boundary of $\partial (F_k'\times (0,1))\setminus(F_k'\cup E_k'$)). Then $(\G_k')^t$ overflows $G_k$, and the same way we also have $T_{\eps}((\G_k')^t)$ overflows $T_{\eps}(G_k)$ and therefore for every $\eps>0$ we have
\begin{align*}
\m T_{\eps}((\G_k')^t) \leq \m T_{\eps}((G_k')^t).
\end{align*}

Next,  we let $\eps_k = 2^{-k}$ and estimate $\m T_{\eps_k}(G_k)$ from above.
Considering the conformal mapping
$$f_k(z)=\frac{1}{\eps_k}(z-\frac{1}{2^k})$$
and using Lemma \ref{lemma:slits} we obtain that
\begin{align*}
\m T_{\eps_k}(G_k) = \m T_{\eps_k}(f_k(G_k))\leq \frac{\eps_k}{1/2} + \frac{1}{2\cdot (2^k/2)} =3 \eps_k.
\end{align*}

Now, since $T_{\eps_k}(G_k) = (T_{\eps_k}(G_k'))^t$, we obtain
\begin{align*}
\eps_k\m(T_{\eps_k}(\G_k')) = \eps_k \cdot\frac{1}{\m(T((\G_k')^t))} \geq \frac{\eps_k}{\m(T_{\eps_k}(G_k))}\geq \frac{\eps_k}{3\eps_k} = \frac{1}{3}.
\end{align*}

\end{proof}

\section{From integrable holomorphic quadratic differentials to bounded measured laminations}

Let $\varphi$ be an integrable holomorphic quadratic differential on
$\mathbb{D}$ (i.e. a holomorphic function $\varphi :\mathbb{D}\to\mathbb{C}$
such that $\|\varphi\|_{L^1(\mathbb{D})}=\iint_{\mathbb{D}}|\varphi
(z)|dxdy<\infty$). Let ${A}(\mathbb{D})$ be the space of all
integrable holomorphic quadratic differentials on $\mathbb{D}$.

Given $\varphi\in{A}(\mathbb{D})$, we defined a corresponding bounded
measured lamination
$$
\mu_{\varphi}([a,b]\times [c,d])=\m(\Gamma_v ([a,b],[c,d]))
$$
or equivalently
$$
\mu_{\varphi}([a,b]\times [c,d])=\int_I \frac{1}{l(z)}|\sqrt{\varphi (z)}dz|
$$
where $I$ is transverse arc to $\Gamma_v ([a,b],[c,d])$.

It follows that if $c>0$ then $\mu_{c\varphi}=\mu_{\varphi}$. Therefore we
obtain a map from the space $P{A}(\mathbb{D})$ of projective
integrable holomorphic quadratic differentials to the space of projective bounded
measured laminations $PML_{bdd}(\mathbb{D})$,
$$
\mathcal{M}:P{A}(\mathbb{D})\to PML_{bdd}(\mathbb{D}).
$$
We prove that $ \mathcal{M}:P\mathcal{A}(\mathbb{D})\to
PML_{bdd}(\mathbb{D}) $ is injective.

\begin{theorem}
\label{thm:1-1} The map $$ \mathcal{M}:P{A}(\mathbb{D})\to
PML_{bdd}(\mathbb{D}).
$$
 defined by
$$
\mathcal{M} ([\varphi] )=[\mu_{\varphi}]
$$
is injective.
\end{theorem}

\begin{proof}
We assume that
\begin{equation}
\label{eq:vv'}
\mu_{\varphi}=c_1\mu_{\varphi'}
\end{equation}
and need to prove that $\varphi
=c\varphi'$ for some $c>0$.
Since $\mu_{\varphi}=c_1\mu_{\varphi'}$ we have that
their geodesic laminations supports $|\mu_{\varphi}|$ and $|c_1\mu_{\varphi'}|$
are the same. In other words each leaf of the vertical foliation $\varphi$ is
homotopic to a leaf of the vertical foliation of $\varphi'$ relative their
two endpoints on the unit circle, and vice versa.

Additionally, assume that the corresponding leaves of the vertical foliations
are not only homotopic but that they are equal to each other. In other words,
the vertical foliations of $\varphi$ and $\varphi'$ are equal. If
$z_0\in\mathbb{D}$  is a regular point of both $\varphi$ and $\varphi'$,
denote by $\zeta$ and $\zeta'$ the corresponding natural parameters in a
regular neighborhood $U$ of $z_0$. Then $f =\zeta'\circ \zeta^{-1}$  is a
conformal mapping from $\zeta (U)\subset\mathbb{C}$ onto
$\zeta'(U)\subset\mathbb{C}$ that maps vertical lines onto vertical lines. It
follows then that $\zeta'=a\zeta +b$ for some $a\in\mathbb{R}$. Thus
$d\zeta'^2=a^2d\zeta^2$ and we set $c=a^2$.

We obtained that for each regular point $z_0$ of $\varphi$ and $\varphi'$
there exist a neighborhood $U\ni z_0$ and a constant $c=a^2>0$ such that $\varphi
=c\varphi'$ in $U$. Since the set of regular points of $\varphi$ and
$\varphi'$ is connected and dense in $\mathbb{D}$ then $\varphi =c\varphi'$
in $\mathbb{D}$ and the proof is finished in this case.

It remains to prove that the vertical foliations of $\varphi$ and $\varphi'$
are the same under the assumption that $\mu_{\varphi}=c_1\mu_{\varphi'}$. Let
$\{ S(\beta_i)\}_{i=1}^{\infty}$ be a family of mutually disjoint vertical strips with open
transverse horizontal arcs $\beta_i$ that covers $\mathbb{D}$  up to countably
many vertical trajectories (cf. \cite{Str}). The metric on the horizontal arcs $\beta_i$ is induced by
$|\sqrt{\varphi (z)}dz|$ and we isometrically identify $\beta_i$ with $(0,a_i)$,
where $a_i$ is the length of $\beta_i$. The variable in $(0,a_i)$ is $x$ and
the integration with respect $dx$ corresponds to integration with respect
$|\sqrt{\varphi (z)}dz|$ in $\mathbb{D}$. The arc $(0,a_i)$ is a horizontal
arc in the natural parameter $w=\int  \sqrt{\varphi (z)}dz$ for $\varphi (z)$.

For $\beta_i$, let $S(\beta_i,(0,x))$ be the substrip of $S(\beta_i)$ of
vertical trajectories going through $(0,x)\subset (0,a_i)$. The area of
$S(\beta_i,(0,x))$ is
$$
A^{\varphi}_{\beta_i}(x)=\int_{(0,x)}l^{\varphi}(v_{\beta_i}^{\varphi}(t))dt,
$$
where $v_{\beta_i}^{\varphi}(t)$ is the vertical trajectory of $\varphi$
through the point $t\in (0,x)\subset \beta_i$ and $l^{\varphi}(\cdot )$ is
the length in the $|\sqrt{\varphi (z)}dz|$ metric. The modulus of the
vertical trajectories in $S(\beta_i,(0,x))$ is
$$
M_{\beta_i}^{\varphi}(x)=\int_{(0,x)}\frac{1}{l^{\varphi}(v_{\beta_i}^{\varphi}(t))}dt.
$$

If necessary, we multiply $\varphi'$ by a positive constant such that
$\|c_1\varphi'\|_{L^1}=\|\varphi\|_{L^1}$. Since the supports of
$\mu_{\varphi}$ and $\mu_{\varphi'}$ are the same, to each $S(\beta_i,(0,x))$
there corresponds a vertical strip $\tilde{S}(\beta_i,(0,x))$ of vertical
trajectories $v_{\beta_i}^{\varphi'}(t)$ of $\varphi'$ with the same
endpoints on $\mathbb{S}^1$ as $v_{\beta_i}^{\varphi}(t)$. Note that
$v_{\beta_i}^{\varphi'}(t)$ does not necessarily pass through $t\in\beta_i$
or even intersects $\beta_i$.

Let $A_{\beta_i}^{\varphi'}(x)$ and $M_{\beta_i}^{\varphi'}(x)$ denote the
area of $\tilde{S}(\beta_i,(0,x))$ and the modulus of vertical trajectories
of $\varphi'$ in $\tilde{S}(\beta_i,(0,x))$. We have the following lemma.

\begin{lemma}
\label{lem:weak} Let $\beta_i$ be a transverse horizontal arc to a vertical
strip $S(\beta_i)$ isometrically identified with $(0,a_i)$ in the natural
parameter of $\varphi$. Then for a.e. $x\in (0,a_i)$, we have
$$
\frac{d}{dx}M_{\beta_i}^{\varphi'}(x)\leq\frac{\frac{d}{dx} A_{\beta_i}^{\varphi'}(x)}{[l^{\varphi}(v_{\beta_i}^{\varphi'}(x))]^2},
$$
where $l^{\varphi}(\cdot )$ is the $\varphi$-length and
$v_{\beta_i}^{\varphi'}(x)$ is the horizontal trajectory of $\varphi'$ whose
endpoints agree with the endpoints of $v_{\beta_i}^{\varphi}(x)$.
\end{lemma}

\begin{proof}
For $x\in (0,a_i)$ and small $\eps>0$ we denote
$$L_x(\eps) = \inf \{ l^{\varphi}(v_{\beta_i}^{\varphi'}(t)) \, : \, t \in {[x,x+\eps ]} \}.$$
Note that $L_x(\eps )>0$ for $\eps >0$ small enough by the continuity of $\varphi'$. It is possible that $L_x(\eps )=\infty$ if all vertical trajectories of $\varphi'$ close to $v_{\beta_i}^{\varphi'}(x)$ have infinite $\varphi$-length.
The metric $\rho (z)\equiv L_x(\eps)^{-1}$ is admissible for $\tilde{S}(\beta_i,[x,x+\eps])$, where $L_x(\eps)^{-1}=0$ if $L_x(\eps)=\infty$.

Since, $L_x(\eps)$ is non-increasing  it has a limit as $\eps\to0^+$. In
fact, we have
$$L_x(\eps)\xrightarrow[\eps\to0^+]{}l^{\varphi}(v_{\beta_i}^{\varphi'}(t)).$$ To see this, note first that
$L_x(\eps)\leq l^{\varphi}(v_{\beta_i}^{\varphi'}(x))$ and we only need to
estimate the limit from below.

Assume first that $l^{\varphi}(v_{\beta_i}^{\varphi'}(x))<\infty$. Fix $\delta>0$ and choose points
$\xi_0, \ldots, \xi_k \in v_{\beta_i}^{\varphi'}(x)$, so that
\begin{equation}
\label{eq:1}
\sum_{i=1}^k d_{\varphi}(\xi_i,\xi_{i-1}) \geq l^{\varphi}(v_{\beta_i}^{\varphi'}(x)) -\frac{\delta}{2}
\end{equation}
where $d_{\varphi}(\xi_i,\xi_{i-1}) $ is the distance between $\xi_i$ and $\xi_{i-1}$ in the $\varphi$ metric, i.e. the metric induced by $|\sqrt{\varphi (z)}dz|$.

We want to show that for small $\eta$ the curves
$v_{\beta_i}^{\varphi'}(x+\eta)$ have lengths at least
$l^{\varphi}(v_{\beta_i}^{\varphi'}(x))-\delta$. Since the set of vertical
trajectories $S(\beta_i)$ foliates a neighborhood of
$v_{\beta_i}^{\varphi}(x)$ then
$\tilde{S}(\beta_i)$ must foliate a neighborhood of
$v_{\beta_i}^{\varphi'}(x)$ because the separation property of the vertical trajectories of $\varphi$ and $\varphi'$ is a topological property of their endpoints. By choosing small $\eta >0$, we get that a
subarc of $v_{\beta_i}^{\varphi'}(x+\epsilon )$ is within small euclidean distance to
the subarc of $v_{\beta_i}^{\varphi'}(x)$ between $\xi_0$ and $\xi_k$ for all $0<\epsilon <\eta$. Since
$\varphi$ is continuous, it follows that for $\eta >0$ small enough, each
$v_{\beta_i}^{\varphi'}(x+\epsilon )$ for $\epsilon <\eta$ has points
$\xi_0',\ldots ,\xi_k'$ on the $\varphi$-distance less than
$\frac{\delta}{4k}$ from $\xi_0,\ldots ,\xi_k$, respectively.
Therefore by (\ref{eq:1}) we have
$$
l^{\varphi}(v_{\beta_i}^{\varphi'}(x+\eta)) \geq  \sum_{i=1}^k d_{\varphi}(\xi_i' , \xi_{i-1}') \geq \sum_{i=1}^k \left(d_{\varphi}(\xi_i,\xi_{i-1}) - \frac{\delta}{2k}\right)\geq l^{\varphi}(v_{\beta_i}^{\varphi'}(x))-\delta.
$$
Thus $L_x(\epsilon )\geq l^{\varphi}(v_{\beta_i}^{\varphi'}(x))-\delta$ for
all $\epsilon <\eta$ which implies that $\lim_{\epsilon\to 0^{+}}L_x(\epsilon
)=l^{\varphi} (v_{\beta_i}^{\varphi'}(x))$ because $\delta >0$ is arbitrary.

Assume now that $l^{\varphi}(v_{\beta_i}^{\varphi'}(x))=\infty$. If $L_x(\eps )=\infty$ for some $\eps >0$ then $\lim_{\eps\to 0^{+}}L_x(\eps )=l^{\varphi}(v_{\beta_i}^{\varphi'}(x))$. We consider the case when $L_x(\eps )<\infty$ for all $\eps >0$ and need to prove that for every $M>0$ there exist $\eta >0$ such that $L_x(\eps )\geq M$ for all $\eps <\eta$. Choose points
$\xi_0, \ldots, \xi_k \in v_{\beta_i}^{\varphi'}(x)$, so that
\begin{equation}
\label{eq:2}
\sum_{i=1}^k d_{\varphi}(\xi_i,\xi_{i-1}) \geq M+1.
\end{equation}
where $d_{\varphi}(\xi_i,\xi_{i-1}) $ is the distance between $\xi_i$ and $\xi_{i-1}$ in the $\varphi$-metric. For small $\eta >0$ and all $\eps <\eta$, a
subarc of $v_{\beta_i}^{\varphi'}(x+\eps )$ is within small euclidean distance to
the subarc of $v_{\beta_i}^{\varphi'}(x)$ between $\xi_0$ and $\xi_k$. Since
$\varphi$ is continuous, it follows that for $\eta >0$ small enough, each
$v_{\beta_i}^{\varphi'}(x+\epsilon )$ for $\epsilon <\eta$ has points
$\xi_0',\ldots ,\xi_k'$ on the $\varphi$-distance less than
$\frac{1}{2k}$ from $\xi_0,\ldots ,\xi_k$, respectively.
Therefore by (\ref{eq:2}) we have
$$
l^{\varphi}(v_{\beta_i}^{\varphi'}(x+\eta)) \geq \sum_{i=1}^k \left(d_{\varphi}(\xi_i,\xi_{i-1}) - \frac{1}{4k}\right)\geq l^{\varphi}(v_{\beta_i}^{\varphi'}(x))-1\geq M
$$
and then $\lim_{\eps\to 0}L_x(\eps )=l^{\varphi}(v_{\beta_i}^{\varphi'}(x))$.

Thus for a.e. $x\in (0,a_i)$ we have

\begin{eqnarray*}
\frac{d}{dx}M_{\beta_i}^{\varphi'}(x)
&=&\lim_{\eps\to 0^{+}}
\frac{\m \tilde{S}(\beta_i,[x,x+\eps])}{\eps}\\
&\leq& \limsup_{\eps\to 0^{+}} \frac{A_{\beta_i}^{\varphi'}(x+\eps)-
A_{\beta_i}^{\varphi'}(x)}{\eps L^2_x(\eps)} =
\frac{\frac{d}{dx}A_{\beta_i}^{\varphi'}(x)}
{[l^{\varphi}(v_{\beta_i}^{\varphi'}(x))]^2},
\end{eqnarray*}
and the proof is complete.
\end{proof}

Using the above lemma we establish the next lemma which finishes the proof.

\begin{lemma}
If for every $\beta_i$ and a.e. $x\in (0,a_i)$ we have
$M_{\beta_i}^{\varphi}(x) =c_1M_{\beta_i}^{\varphi'}(x) $ then the vertical
foliations of $\varphi$ and $\varphi'$ are equal.
\end{lemma}

\begin{proof}
Note that $M_{\beta_i}^{\varphi}(x)=\mu_{\varphi}((0,x))$ and $M_{\beta_i}^{\varphi'}(x)=\mu_{\varphi'}((0,x))$, and equation (\ref{eq:vv'}) imply that
$$M_{\beta_i}^{\varphi}(x)=c_1M_{\beta_i}^{\varphi'}(x)$$ for every $x\in (0,a_i)$.
By the previous lemma, by Lebesque's differentiation theorem and by absolute continuity of $M_{\beta_i}^{\varphi}(x)=\int_{(0,x)}\frac{1}{l^{\varphi}(v_{\beta_i}^{\varphi}(t))}dt$
we have for a.e. $x\in (0,a_i)$
\begin{eqnarray}\label{ineq:mod-derivative}
\begin{split}
\frac{\frac{d}{dx}A_{\beta_i}^{\varphi}(x)}{[l^{\varphi}(v_{\beta_i}^{\varphi}(x))]^2}=\frac{1}{[l^{\varphi}(v_{\beta_i}^{\varphi}(x))]^2}\lim_{\eps\to 0}\frac{1}{\eps}\int_x^{x+\eps}l^{\varphi}(v_{\beta_i}^{\varphi}(t))dt =\\
\frac{1}{l^{\varphi}(v_{\beta_i}^{\varphi}(x))}
= \frac{d}{dx}M_{\beta_i}^{\varphi}(x)=  c_1\frac{d}{dx}M_{\beta_i}^{\varphi'}(x)\leq \frac {c_1\frac{d}{dx}A_{\beta_i}^{\varphi'}(x)}{[l^{\varphi}(v_{\beta_i}^{\varphi'}(x))]^2}.
\end{split}
\end{eqnarray}

Since $l^{\varphi}(v_{\beta_i}^{\varphi}(x))\leq
l^{\varphi}(v_{\beta_i}^{\varphi'}(x))$ with equality implying that the two
curves are the same, it follows from (\ref{ineq:mod-derivative}) that for a.e. $x\in (0,a_i)$ we have
\begin{align}
\frac{d}{dx}A_{\beta_i}^{\varphi}(x)\leq c_1\frac{d}{dx}A_{\beta_i}^{\varphi'}(x).
\end{align}

Note that that $A_{\beta_i}^{\varphi}(x)$ is absolutely continuous in $x$, since $A_{\beta_i}^{\varphi}(x)=\int_{(0,x)}l^{\varphi}(v_{\beta_i}^{\varphi}(t))dt$.
Thus
\begin{equation*}
A_{\beta_i}^{\varphi}(x)=\int_{0}^x \frac{d}{dt}A_{\beta_i}^{\varphi}(t) dt \leq \int_0^x c_1\frac{d}{dt}A_{\beta_i}^{\varphi'}(t) dt  \leq
c_1A_{\beta_i}^{\varphi'}(x)
\end{equation*}
for a.e. $x\in (0,a_i)$. Since
$$
\|\varphi\|_{L^1}=\sum_iA_{\beta_i}^{\varphi}(a_i)\leq c_1\sum_iA_{\beta_i}^{\varphi'}(a_i)=\|c_1\varphi'\|_{L^1}
$$
and $\|\varphi\|_{L^1}=\|c_1\varphi'\|_{L^1}$, we necessarily have equality for
each term $\beta_i$ and for a.e. $x\in (0,a_i)$.

Thus $A_{\beta_i}^{\varphi}(x)=c_1A_{\beta_i}^{\varphi'}(x)$ for a.e. $x\in
(0,a_i)$ which implies
$\frac{d}{dx}A_{\beta_i}^{\varphi}(x)=c_1\frac{d}{dx}A_{\beta_i}^{\varphi'}(x)$ for a.e. $x\in (0,a_i)$.

Therefore, by (\ref{ineq:mod-derivative}) 
we have
$$\frac{\frac{d}{dx}A_{\beta_i}^{\varphi}(x)}{[l^{\varphi}(v_{\beta_i}^{\varphi}(x))]^2}\leq
\frac{c_1\frac{d}{dx}A_{\beta_i}^{\varphi'}(x)}{[l^{\varphi}(v_{\beta_i}^{\varphi'}(x))]^2},$$
where the numerators are equal for a.e. $x\in(0,a)$. In particular, $l^{\varphi}(v_{\beta_i}^{\varphi'}(x))\leq l^{\varphi}(v_{\beta_i}^{\varphi}(x))$ and thus 
$l^{\varphi}(v_{\beta_i}^{\varphi'}(x))=l^{\varphi}(v_{\beta_i}^{\varphi}(x))$
for a.e. $x$. By the uniqueness of geodesics in the $\varphi$ metric connecting two boundary points of simply connected domains (cf. \cite{Str}) and since vertical trajectories foliate $\mathbb{D}$, we
obtain that all vertical trajectories of $\varphi$ and $\varphi'$ are the
same.
\end{proof}

The above lemma together with the above finishes the proof of the theorem.
\end{proof}

Given an integrable holomorphic quadratic differential $\varphi$ on the unit
disk, we denote by $\nu_{\varphi}$ the measured lamination whose support $v_{\varphi}$ is
homotopic to the leaves of the vertical foliation of $\varphi$ and the
transverse measure is given by $\int_I |\sqrt{\varphi (z)}dz|$, where $I$
is a horizontal arc intersecting the leaves of the vertical foliation corresponding to
the leaves of $\nu_{\varphi}$. We prove that $\nu_{\varphi}$ is
Thurston bounded.

\begin{proposition}
Let $\varphi$ be an integrable holomorphic quadratic differential on the unit
disk $\mathbb{D}$. Then the vertical foliation measure $\nu_{\varphi}$
defined above is Thurston bounded.
\end{proposition}

\begin{proof}
Let $\mathcal{V}^{\geq 1}$ be the the set of all vertical trajectories of
$\varphi$ whose $\varphi$-length is $\geq 1$. Let $\mathcal{V}^{< 1}$ be the
the set of all vertical trajectories of $\varphi$ whose $\varphi$-length is
$< 1$. Let $\mathbb{D}^{\geq 1}= \cup_{\gamma \in \mathcal{V}^{\geq 1}}
\gamma$ and $\mathbb{D}^{< 1}= \cup_{\gamma \in \mathcal{V}^{< 1}} \gamma$.

Let $[a,b]\times [c,d]\subset (\mathbb{S}^1\times \mathbb{S}^1)-diag$ be a
box of geodesics with $\mathcal{L}([a,b]\times [c,d])=\log 2$. This implies
that $\frac{1}{C}\leq \m([a,b],[c,d];\mathbb{D})\leq C$ for some $C>1$. Let
$I$ be at most countable union of horizontal arcs of $\varphi$ that
intersects each vertical trajectory of $\varphi$ in exactly one point. Then
we have
$$
\|\varphi\|_{L^1}>\iint_{\mathbb{D}^{\geq 1}}|\varphi (z)|dxdy=\int_{I\cap \mathbb{D}^{\geq 1}}l^{\varphi}(v^{\varphi}(z))dx\geq\int_{I\cap \mathbb{D}^{\geq 1}}dx,
$$
where $x$ is the real part of the natural parameter along $I$.

For a box of geodesics $[a,b]\times [c,d]$, let $\mathcal{V}^{<1}_{[a,b]\times [c,d]}$ be the set of vertical trajectories of length less than $1$ that connects $[a,b]$ and $[c,d]$. Let $I^{<1}_{[a,b]\times [c,d]}$ be the subset of $I$ that intersects only vertical trajectories of $\mathcal{V}^{<1}_{[a,b]\times [c,d]}$.
Then we
have
\begin{equation*}
\begin{split}
C\geq \m([a,b],[c,d];\mathbb{D})\geq \m(\mathcal{V}^{< 1}_{[a,b]\times [c,d]})=\\ \int_{I^{< 1}_{[a,b]\times [c,d]}}\frac{1}{l^{\varphi}(v^{\varphi}(z))}dx\geq \int_{I^{< 1}_{[a,b]\times [c,d]}}dx.
\end{split}
\end{equation*}
Let $I_{[a,b]\times [c,d]}$ be the subset of $I$ that intersects only the vertical trajectories of $\varphi$ that connect $[a,b]$ to $[c,d]$.
Since $\nu_{\varphi} ([a,b]\times [c,d])=\int_{I_{[a,b]\times [c,d]}}dx\leq \|\varphi\|_{L^1}+C$,
we have that $\|\nu_{\varphi}\|_{Th}<\infty$.
\end{proof}

\vskip .2 cm

We can recover the integral $\|\varphi\|_{L^1}$ from $\mu_{\varphi}$ and $\nu_{\varphi}$ by the formula
$$
\|\varphi\|_{L^1}=\int_{\mathbb{S}^1\times \mathbb{S}^1-diag}\frac{d\mu_{\varphi}}{d\nu_{\varphi}}d\mu_{\varphi}.
$$
This immediately gives

\begin{theorem}
\label{thm:injective}
The map from $A(\mathbb{D})$ in $ML_{bdd}(\mathbb{D})\times ML_{bdd}(\mathbb{D})$ given by
$$
\varphi\mapsto (\nu_{\varphi},\mu_{\varphi})
$$
is injective.
\end{theorem}


\begin{thebibliography}{Thua}

\vskip .5cm


\bibitem{Ahl} Ahlfors, Lars V.,\textit{Conformal
    invariants: topics in geometric function theory.}
    McGraw-Hill Series in Higher Mathematics. McGraw-Hill Book Co.,
    New York-Düsseldorf-Johannesburg, 1973. ix+157 pp.

\bibitem{Bon1} F. Bonahon, {\it The geometry of Teichm\"uller space via
    geodesic currents}, Invent. Math. \textbf{92} (1988), no. 1, 139 -- 162.

\bibitem{CMW} J. Chaika, H. Masur and M. Wolf, {\it Limits in PMF of Teichm\" uller geodesics},  	arXiv:1406.0564 [math.GT].

\bibitem{GL} F. Gardiner and N. Lakic, {\it Quasiconformal Teichm\"uller
    theory.} Mathematical Surveys and Monographs, 76. American Mathematical
    Society, Providence, RI, 2000. xx+372 pp.

\bibitem{GM} J. Garnett and D. Marshall, {\it Harmonic
    measure.} New Mathematical Monographs, 2. Cambridge University Press,
    Cambridge, 2005. xvi+571 pp.

\bibitem{FLP} A. Fathi, F. Laudenbach and V. Po\' enaru, {\it Thurston's work
    on surfaces}, Translated from the 1979 French original by Djun M. Kim and
    Dan Margalit. Mathematical Notes, 48. Princeton University Press,
    Princeton, NJ, 2012.

\bibitem{HaSar} H. Hakobyan and D. \v Sari\' c, {\it Vertical limits of
    graph domains}. To appear in  Proc. Amer. Math. Soc.

\bibitem{Hein} J. Heinonen, {\it Lectures on analysis on metric spaces.}
    Universitext. Springer-Verlag, New York, 2001. x+140 pp.

\bibitem{Kei} S. Keith, {\it Modulus and the Poincar\' e inequality on metric
    measure spaces} Math. Z. \textbf{245} (2003), no. 2, 255�-292.

\bibitem{LV}  O. Lehto and K. I. Virtanen, {\it  Quasiconformal mappings in
    the plane.}, Second edition. Translated from the German by K. W. Lucas.
    Die Grundlehren der mathematischen Wissenschaften, Band 126.
    Springer-Verlag, New York-Heidelberg, 1973.

\bibitem{LLR} C. Leininger, A. Lenzhen and K. Rafi, {\it Limit sets of
    Teichm\"uller geodesics with minimal non-uniquely ergodic vertical
    foliation}. {\tt arXiv:1312.2305[math.GT]}

\bibitem{Len} A. Lenzhen, {\it Teichm\"uller geodesics that do not have a
    limit in PMF}, Geom. Topol. \textbf{12} (2008), no. 1, 177--197.

\bibitem{Mas} H. Masur, {\it Two boundaries of Teichm\"uller space}, Duke
    Math. J. \textbf{49} (1982), no. 1, 183--190.

\bibitem{MiSar} H. Miyachi and D. \v Sari\' c, {\it Uniform weak $*$
    topology and earthquakes in the hyperbolic plane.}
    Proc. Lond. Math. Soc. (3) \textbf{105} (2012), no. 6, 1123-�1148.

\bibitem{Ot} J. P. Otal, {\it About the embedding of Teichm\" uller space in
    the space of geodesic H\" older distributions. Handbook of Teichm\" uller
    theory}, Vol. I, 223--248, IRMA Lect. Math. Theor. Phys., 11, Eur. Math.
    Soc., Z�rich, 2007.

\bibitem{Pomm} Ch. Pommerenke, {\it  Boundary behaviour of conformal maps},
    Grundlehren der Mathematischen Wissenschaften [Fundamental Principles of
    Mathematical Sciences], 299. Springer-Verlag, Berlin, 1992.

\bibitem{Sa1} D. \v Sari\' c, {\it Real and Complex Earthquakes}, Trans.
    Amer. Math. Soc. \textbf{358} (2006), no. 1, 233--249.


\bibitem{Sa2}  D. \v Sari\' c, {\it Geodesic currents and Teichm\"uller
    space}, Topology \textbf{44} (2005), no. 1, 99--130.

\bibitem{Sa4} D. \v Sari\' c, {\it Infinitesimal Liouville distributions for
    Teichm\"uller space}, Proc. London Math. Soc. (3) \textbf{88} (2004), no. 2,
    436-454.

\bibitem{Sa3} D. \v Sari\' c, {\it Thurston's boundary for Teichm\"uller
    spaces of infinite surfaces: geodesic currents}, Preprint, available on arXiv.org.

\bibitem{Str} K. Strebel, {\it Quadratic differentials.} Ergebnisse der
    Mathematik und ihrer Grenzgebiete (3) [Results in Mathematics and Related Areas
    (3)], 5. Springer-Verlag, Berlin, 1984.

\bibitem{Vaisala:lectures} J. $\rm{V\ddot{a}is\ddot{a}l\ddot{a}}$,
    \emph{Lectures on n-dimensional quasiconformal mappings.} Lecture Notes
    in Mathematics, \textbf{229}. Springer-Verlag, Berlin-New York, 1971.
\end{thebibliography}
\end{document}